\documentclass[12pt,reqno]{amsart}

\numberwithin{equation}{section}

\usepackage{amsbsy}

\usepackage[colorlinks]{hyperref}
\hypersetup{
linkcolor=blue,          
citecolor=green,        
}
\usepackage[nobysame,abbrev,alphabetic]{amsrefs}

\usepackage{enumerate}
\usepackage{amssymb}
\usepackage{amsmath}
\usepackage{amscd}
\usepackage{amsthm}
\usepackage{amsfonts}
\usepackage{graphicx}
\usepackage[all]{xy}
\usepackage{verbatim}
\usepackage{hyperref}
\usepackage{gensymb}
\usepackage{mathrsfs}

\newtheorem{theorem}{Theorem}[section]

\newtheorem{lemma}[theorem]{Lemma}
\newtheorem{corollary}[theorem]{Corollary}
\newtheorem{cor}[theorem]{Corollary}
\theoremstyle{definition}\newtheorem{definition}[theorem]{Definition}

\newtheorem{proposition}[theorem]{Proposition}

\theoremstyle{definition}

\theoremstyle{definition}
\theoremstyle{definition}\newtheorem{remark}[theorem]{Remark}
\theoremstyle{definition}

\newcommand{\al}{\alpha}

\newcommand{\ga}{\gamma}
\newcommand{\Ga}{\Gamma}
\newcommand{\del}{\delta}
\newcommand{\Del}{\Delta}

\newcommand{\Lam}{\Lambda}

\newcommand{\Om}{\Omega}
\newcommand{\vphi}{\varphi}
\newcommand{\vre}{\varepsilon}
    \newcommand{\Xn}{{\mathcal{L}_n}}
\newcommand\crly[1]{\mathscr{#1}}
\newcommand{\sm}{\smallsetminus}

\newcommand{\df}{{\, \stackrel{\mathrm{def}}{=}\, }}

\newcommand\Name[1]{\label{#1}{\ifdraft{\sn
      [#1]}\else\ignorespaces\fi}}

\newcommand\eq[2]{{\ifdraft{\ \tt
      [#1]}\else\ignorespaces\fi}\begin{equation}\label{#1}{#2}\end{equation}} 
\newcommand {\equ}[1]{\eqref{#1}}

\newcommand{\BB}{{\mathcal{B}}}

\newcommand\Vol{\mathrm{Vol}}

\newcommand\covrad{\mathrm{covrad}}
\newcommand\conv{\mathrm{conv}}
\newcommand\supp{\mathrm{supp}}

\newcommand{\cA}{\mathcal{A}}
\newcommand{\cB}{\mathcal{B}}

\newcommand{\cD}{\mathcal{D}}
\newcommand{\cE}{\mathcal{E}}

\newcommand{\cI}{\mathcal{I}}

\newcommand{\cL}{\mathcal{L}}
\newcommand{\cM}{\mathcal{M}}

\newcommand{\cS}{\mathcal{S}}

\newcommand{\cU}{\mathcal{U}}
\newcommand{\cV}{\mathcal{V}}
\newcommand{\cW}{\mathcal{W}}

\newcommand{\cY}{\mathcal{Y}}

\newcommand{\spa}{{\rm span}}

\newcommand{\LL}{\mathcal{L}}

\newcommand{\bR}{\mathbb{R}}
\newcommand{\bZ}{\mathbb{Z}}

\newcommand{\R}{{\mathbb{R}}}

\newcommand{\Z}{{\mathbb{Z}}}

\newcommand {\ignore}[1]  {}

\newcommand{\SL}{\operatorname{SL}}

\newcommand{\defi}{\overset{\on{def}}{=}}

\newcommand\norm[1]{\left\|#1\right\|}

\newcommand\set[1]{\left\{#1\right\}}
\newcommand\pa[1]{\left(#1\right)}

\newcommand{\E}{\mathbf{e}}

\newcommand\av[1]{\left|#1\right|}
\newcommand\on[1]{\operatorname{#1}}

\newcommand\tb[1]{\textbf{#1}}
\newcommand\mat[1]{\pa{\begin{matrix}#1\end{matrix}}}
\newcommand\br[1]{\left[#1\right]}
\newcommand\smallmat[1]{\pa{\begin{smallmatrix}#1\end{smallmatrix}}}

\newcommand{\lra}{\longrightarrow}

\newcommand{\onto}{\xymatrix{\ar@{>>}[r]&}}
\newcommand{\da}[4]{\xymatrix{#1 \ar@<.5ex>[r]^{#2} \ar@<-.5ex>[r]_{#3} & #4}}

\newif\ifdraft\drafttrue
\draftfalse

\font\sn = cmssi8 scaled \magstep0

\newcommand{\nerve}{{\rm Nerve}}
\newcommand{\order }{{\mathrm {ord}}}
\newcommand{\Lb}{{\mathrm {Leb}}}
\newcommand{\mesh}{{\mathrm {mesh}}}

\newcommand{\asdim}{{\rm asdim}}

\begin{document}
\title{On stable lattices and the diagonal group}
\author{Uri Shapira}
\address{Dept. of Mathematics, Technion, Haifa, Israel
{\tt put email here}
}
\author{Barak Weiss}
\address{Dept. of Mathematics, Tel Aviv University, Tel Aviv, Israel
{\tt barakw@post.tau.ac.il}}

\maketitle
\begin{abstract}
Inspired by work of McMullen, we show that any orbit of the diagonal
group in the space of lattices accumulates on the set of stable
lattices. As consequences, we settle a conjecture of Ramharter
concerning the asymptotic behavior of the Mordell constant, and
reduce Minkowski's conjecture on products of linear forms to a
geometric question, yielding two new proofs of the conjecture in
dimensions up to 7. We also answer a question of Harder on the volume
of the set of stable lattices.
\end{abstract}
\section{Introduction}
Let $n \geq 2$ be an integer, let $G \df \SL_n(\R), \, \Gamma \df
\SL_n(\Z)$, let $A \subset G$ be the subgroup of positive diagonal
matrices and let  $\Xn \df G/\Gamma$ be the space of unimodular 
lattices in $\R^n$. The purpose of this paper is to present a
dynamical result regarding the action of $A$ on $\Xn$, and to present
some consequences in the geometry of numbers. 

A lattice $x \in \Xn$ is called {\em stable} if for any
subgroup $\Lambda \subset x$, the covolume of $\Lambda$ in
$\spa(\Lambda)$ is at least 1. In particular the length of the
shortest nonzero vector in $x$ is at least 1. Stable lattices have
also been called `semistable', they were introduced in a broad
algebro-geometric context by Harder, Narasimhan and Stuhler
\cite{Stuhler, Harder}, and were used to develop a
reduction theory for the study of the topology of locally symmetric
spaces. See Grayson \cite{Grayson} for a clear exposition. 
\begin{theorem}\Name{thm: main} 
For any $x \in \Xn$, the orbit-closure $\overline{Ax}$ contains a
stable lattice. 
\end{theorem}

Theorem \ref{thm: main} is inspired by
a breakthrough result of McMullen \cite{McMullenMinkowski}. Recall that a lattice in $\Xn$ is
called {\em well-rounded} if its shortest nonzero vectors span $\R^n$.
In connection with his work on
Minkowski's conjecture, McMullen showed that the closure of any bounded
$A$-orbit in $\Xn$ contains a well-rounded lattice. The
set of well-rounded lattices neither contains, nor is contained in,
the set of stable lattices,
but the proof of Theorem \ref{thm:
  main} closely follows McMullen's strategy.

We apply Theorem \ref{thm: main} to two problems in the geometry of
numbers. 
Let $x \in \Xn$ be a unimodular lattice. By a {\em symmetric box} in
$\R^n$ we mean a set of the form $ [-a_1, a_1] \times \cdots \times
[-a_n, a_n]$, and we say that a symmetric box is 
 {\em admissible} for $x$ if it contains no nonzero points of
 $x$ in its interior.  The {\em Mordell constant} of $x$
 is defined to be  
\eq{eq: defn const}{
\kappa(x) \df \frac{1}{2^n} \sup_{\BB} \Vol(\BB), 
}
where the supremum is taken over admissible symmetric boxes $\BB$, and
where $\Vol(\BB)$ denotes the volume of $\BB$.
We also write 
\eq{eq: defn kappan}{\kappa_n \df \inf\{\kappa(x): x \in \Xn\}.
}
The infimum in this definition is in fact a minimum, and, as with many
problems in the geometry of numbers it is of interest to compute the
constants $\kappa_n$ and identify the lattices realizing the
minimum. However this appears to be a very difficult problem, which so
far has only been solved for $n=2,3$, the latter in a difficult paper
of Ramharter \cite{Ramharter_dim3}. It is also of interest to provide
bounds on the asymptotics of $\kappa_n$, and in \cite{Ramharter_conjecture},
Ramharter conjectured that $\limsup_{n \to \infty} \kappa_n^{1/n\log
  n}>0$. As a simple corollary of Theorem \ref{thm: main}, we validate
Ramharter's conjecture, with an explicit bound: 
\begin{cor}\Name{cor: Ramharter conj}
For all $n \geq 2,$ 
\eq{eq: our bound}{\kappa_n \geq n^{-n/2}.
} 
In particular
$$
\kappa_n^{1/n\log
  n}  \geq n^{-1/2\log n}  \longrightarrow_{n \to \infty} 
 \frac{1}{\sqrt{e}}. 
$$
\end{cor}
We remark that Corollary \ref{cor: Ramharter conj} could also be
derived from McMullen's results and a theorem of Birch and
Swinnerton-Dyer. In \S \ref{sec: rankin
  bounds} we show that the bound \equ{eq: our bound} is not 
optimal and explain how to obtain better bounds, for all $n$ which are
not divisible by 4.  We refer the reader to
\cite{gruber} for more information on the possible values of
$\kappa(x), x \in \Xn$. 

Our second application concerns Minkowski's conjecture\footnote{It is not clear to us whether
  Minkowski actually made this conjecture.}, 
which posits that for any unimodular lattice $x$, one has 
\eq{eq: Minkowski conj}{
\sup_{u \in \R^n} \, \inf_{v \in x} |N(u-v)| \leq \frac{1}{2^n},
}
where $N(u_1, \ldots, u_d) \df \prod_j u_j.$
Minkowski solved the question for $n=2$ and several
authors resolved the cases $n \leq 5$. 
In \cite{McMullenMinkowski}, McMullen settled the case
$n=6$. In fact, using his theorem on the $A$-action on $\Xn$, McMullen
showed that in arbitrary dimension $n$, 
Minkowski's conjecture is implied by the statement that any well-rounded
lattice $x \subset \R^d$ with $d \leq n$ satisfies 
\eq{eq: covrad}{\covrad(x) \leq \frac{\sqrt{d}}{2},}
where $\covrad(x) \df \max_{u \in \R^d} \min_{v \in x} \|u-v\|$
and $\| \cdot \|$ is the Euclidean norm on $\R^d$. At the time of
writing \cite{McMullenMinkowski}, \equ{eq:
  covrad} was known to hold for well-rounded lattices in dimension at most $6$, and in recent work of
Hans-Gill, Raka, Sehmi and Leetika \cite{hans-gill1, hans-gill2, leetika}, \equ{eq: covrad} has
been proved for well-rounded lattices in dimensions $n=7,8,9$, thus settling Minkowski's question in
those cases. 

Our work gives two new  approaches to Minkowski's conjecture. A direct
application of Theorem \ref{thm: main} (see Corollary 
\ref{cor: for Minkowski 1}) shows that it follows in 
dimension $n$, from the assertion that for any stable $x \in \Xn$, \equ{eq: covrad} holds. Note
that we do not require \equ{eq: covrad} in dimensions less than
$n$. Using the strategy of Woods and Hans-Gill et al, in Theorem
\ref{thm: use of KZ diagonal} we define a compact subset  $\mathrm{KZS} \subset \R^n$ and a
collection 
of $2^{n-1}$ subsets $\{ \cW(\cI)\}$ of $ \R^n$. We show that the
assertion $\mathrm{KZS} \subset \bigcup_{\cI} \cW(\cI)$ implies
Minkowski's conjecture in dimension 
$n$.

Secondly, an induction using the naturality of stable lattices,
leads to the following sufficient condition: 

\begin{cor}\Name{cor: Minkowski}
Suppose that for some dimension $n$, for all $d\leq n$, any stable
lattice $x \in \LL_{d}$ which is a local maximum of the function covrad, satisfies
\equ{eq: covrad}. Then \equ{eq: Minkowski conj} holds for any $x \in \Xn$. 
\end{cor}
The local maxima of the function covrad have been
studied in depth in recent work of Dutour-Sikiri\'c, Sch\"urmann
and Vallentin \cite{mathieu}, who characterized them and showed that there
are finitely many  in each dimension. 
These two approaches give two new
proofs of Minkowski's Conjecture  in dimensions $n \leq 7$.

A natural question is to what extent stable lattices are typical
in $\Xn$. The definition of stability may appear at first sight to be very
restrictive. Nevertheless in \S \ref{sec: volume computation} we show that as $n\to \infty$,
the probability that a random lattice is stable tends to 1 (where the
probability is taken 
with respect to the
natural $G$-invariant measure on $\Xn$). This answers a question of
G. Harder. In fact a stronger statement is true, see Proposition
\ref{prop: strengthening volume}. For the results of \S \ref{sec: volume computation}
we use Siegel's approach to measure volumes \cite{SiegelFormula} and
rely on computations of Thunder \cite{Thunder}.

A significant difference between Theorem \ref{thm:
  main} and McMullen's work on well-rounded lattices, is that we do
not need to assume that the orbit $Ax$ is bounded.    
One may wonder whether McMullen's result is valid without the
hypothesis that $Ax$ is bounded, namely is it true that the closure of
any $A$-orbit in $\Xn$ (or even the orbit itself) contains a
well-rounded lattice? We answer these 
questions affirmatively for {\em closed} orbits in \S \ref{sec: closed
orbits}.  To this end we use results
of Tomanov and the second-named author \cite{TW}, as well as a
covering result (which we learned from  Michael Levin) generalizing one of the results of
\cite{McMullenMinkowski}; this topological result appears to be
well-known to experts, but in order to keep this paper self-contained,
we give the proof in the appendix. For
another perspective on this and related questions, see \cite{PS}. 

\subsection{Acknowledgements} Our work was inspired by Curt McMullen's
breakthrough paper \cite{McMullenMinkowski}
and many of our arguments are adaptations of
arguments appearing in \cite{McMullenMinkowski}. 
We are also grateful to Curt McMullen for additional insightful remarks, and in
particular for the suggestion to study the set of stable lattices in
connection with the $A$-action on $\Xn$. We thank Michael Levin for
useful discussions on topological questions and for agreeing to
include the proof of Theorem \ref{thm: covering} in the appendix. We also thank Mathieu
Dutour-Sikiri\'c, Rajinder Hans-Gill, G\"unter Harder, Gregory Minton and
Gerhard Ramharter for useful discussions. The authors' work was
supported by ERC starter grant DLGAPS 279893 and ISF grant 190/08.

\section{Orbit closures and stable lattices}
Given a lattice $x\in \Xn$ and a subgroup $\Lam \subset x$, we denote by
$r(\Lam)$ the rank of $\Lam$ and by $\av{\Lam}$ the covolume of $\Lam$
in the linear subspace $\spa(\Lambda)$. 
Let 
\begin{align}\label{alpha}
\nonumber \cV(x)&\defi\set{\av{\Lam}^{\frac{1}{r(\Lam)}}: \Lam \subset
  x 
}, \\
\al(x)&\defi\min\cV(x).
\end{align}
Since we may take $\Lam = x$ we have 
$\alpha(x) \leq 1$ for all $x \in \Xn$, and $x$ is stable precisely if $\alpha(x)=1$. 
Observe that $\cV(x)$ 
is a countable discrete subset of the positive reals, and hence the
minimum in 
\eqref{alpha} is attained. 
Also note that the function $\al$ is a variant of the `length of the shortest
  vector'; it is continuous and the sets $\{x: \alpha(x) \geq \vre\}$
  are an exhaustion of $\Xn$ by compact sets. 

We begin by explaining the strategy for proving Theorem \ref{thm:
  main}, which is identical to the one used by McMullen. 
For a lattice $x\in X$ and $\vre>0$  we define an open cover  
$\cU^{x,\vre}=\set{U^{x,\vre}_k}_{k=1}^n$ 
of the diagonal group $A$, where if $a\in U^{x,\vre}_k$ then $\al(ax)$
is `almost attained' by a subgroup of rank $k$. In particular,  
if $a\in U^{x,\vre}_n$ then $ax$ is `almost stable'. 
The main point is to show that for any $\vre>0$, $U^{x,\vre}_n \neq
\varnothing$; for then, taking $\vre_j \to 0$ and $a_j \in A$ such
that $a_j\in U_n^{x,\vre_j}$, we find (passing to a subsequence) that
$a_jx$ converges to a stable lattice. 

In order to establish that
$U_n^{x,\vre}\ne\varnothing$, we apply a topological result of McMullen
(Theorem~\ref{topological input}) regarding open covers  
which is reminiscent of the classical result of Lebesgue
that asserts that in an open cover of Euclidean $n$-space by bounded balls
there must be a point which is covered $n+1$ times. We will work to
show that the 
cover $\cU^{x,\vre}$ satisfies the assumptions of
Theorem~\ref{topological input}. We will be able to verify these assumptions
when the orbit $Ax$ is bounded. In~\S\ref{sec: reduction to compact orbits} we reduce the proof of
Theorem~\ref{thm: main} to this case.

\subsection{Reduction to bounded orbits}\Name{sec: reduction to compact orbits}
Using a result of Birch and Swinnerton-Dyer, we
will now show that it suffices to prove  
Theorem~\ref{thm: main}
under the assumption that the orbit $Ax\subset \Xn$ is bounded; that is,
that $\overline{Ax}$ is compact. 
In this subsection we will denote $A,G$ by $A_n, G_n$ as various dimensions will appear. 

For a matrix $g\in G_n$ we denote by $\br{g}\in \Xn$ the corresponding lattice. If 
\eq{block form}{
g=\mat{g_1&*&\dots&*\\ 0& g_2&\dots&\vdots \\ \vdots& &\ddots&* \\ 0&\dots&0&g_k}
}
where $g_i\in G_{n_i}$ for each $i$, then we say that $g$ is in
\textit{upper triangular block form} 
and refer to the $g_i$'s as the \textit{diagonal blocks}. Note 
that in this definition, we insist that
each $g_i$ is of determinant one. 

\begin{lemma}\Name{lem: block stable is stable}
Let $x=\br{g}\in \Xn$  where $g$ is in upper triangular block form as
in~\eqref{block form} and for each $1\le i\le k$, $\br{g_i}$ is  
a stable lattice in $\mathcal{L}_{n_i}$. Then $x$ is stable.
\end{lemma}
\begin{proof}
By induction, in proving the Lemma we may assume that $k=2$. Let us
denote the standard basis of $\R^n$ by $\E_1, \ldots, \E_n$, let
us write $n=n_1+n_2$, 
$V_1 \df\on{span}\set{\E_1, \ldots, \E_{n_1}}$,
$V_2 \df \on{span}\set{\E_{n_1+1} \ldots, \E_n}$, and let $\pi: \R^n
\to V_2$ be the natural projection. By construction we have $x \cap
V_1 = [g_1], \pi(x) = [g_2]$. 
 
Let $\Lam \subset x$ be a subgroup, write 
$\Lam_1 \df \Lam\cap V_1$ and choose a direct complement 
$\Lam_2 \subset \Lam$, that is 
$$\Lam=\Lam_1+\Lam_2, \ \ \Lam_1 \cap \Lam_2 = \{0\}.$$ 
We claim that 
\eq{eq: claim 1}{\av{\Lam}=\av{\Lam_1}\cdot\av{\pi(\Lam_2)}.}
To see this we recall that one may compute
$|\Lam|$ via the Gram-Schmidt process. Namely, one begins with a set of generators $v_j$
of $\Lambda$ and successively defines $u_1=v_1$ and $u_j$ is the
orthogonal projection of $v_j$ on $\spa (v_1, \ldots,
v_{j-1})^\perp$. In these terms, $|\Lam| = \prod_j \|u_j\|$. Since $\pi$ is an orthogonal
projection and $\Lam \cap V_1$ is in $\ker \pi$, \equ{eq: claim 1} is clear from the
above description. 

The discrete subgroup
$\Lam_1$, when viewed as a subgroup of $\br{g_1}\in \mathcal{L}_{n_1}$ satisfies
$\av{\Lam_1}\ge 1$ because $\br{g_1}$ is assumed to be
stable. Similarly $\pi(\Lam_2) \subset [g_2] \in \mathcal{L}_{n_2}$
satisfies $\av{\pi(\Lam_2)}\ge 1$, hence 
$\av{\Lam}\ge 1$. 
\end{proof}
\begin{lemma}\Name{lem: compt red}
Let $x\in \Xn$ and assume that $\overline{Ax}$ contains a lattice
$\br{g}$ with $g$ of upper triangular block form 
as in~\eqref{block form}. For each $1\le i\le k$,  
suppose $\br{h_i}\in\overline{A_{n_i}\br{g_i}}\subset \mathcal{L}_{n_i}$. Then there
exists a lattice $\br{h}\in\overline{Ax}$ such that $h$ has the
form~\eqref{block form} with $h_i$ as its diagonal blocks. 
\end{lemma}
\begin{proof}
Let $\Omega$ be the set of all lattices $[g]$ of a fixed triangular
form as in ~\eqref{block form}. Then $\Omega$ is a closed subset of
$\Xn$ and there is a projection 
$$\tau: \Omega \to \mathcal{L}_{n_1} \times
\cdots \times \mathcal{L}_{n_k}, \ \ \tau(\br{g}) =
(\br{g_1},\dots,\br{g_k}).$$ 
The map $\tau$ has a compact fiber and is equivariant with respect to the action
of $\widetilde{ A} \df A_{n_1} \times \cdots \times A_{n_k}$. 
By assumption, there is a sequence  $\tilde{a}_j = \left(a^{(j)}_1,
\ldots, a^{(j)}_k\right), \ a^{(j)}_i \in A_{n_i}$ in
$\widetilde{ A}$ such that $a^{(j)}_i [g_i] \to [h_i]$, then after passing to
a subsequence, $\tilde{a}_j [g] \to [h]$ where $h$ has the required
properties. Since $\overline{Ax} \supset \overline {A[g]}$, the claim
follows. 
\end{proof}
\begin{lemma}\label{BSD}
Let $x\in \Xn$. Then there is $[g] \in \overline{Ax}$ such that, up to
a possible permutation of the coordinates,
$g$ is of upper triangular block form as in~\eqref{block form} and 
each $A_{n_i}\br{g_i}\subset \mathcal{L}_{n_i}$ is bounded.
\end{lemma}
\begin{proof}
If the orbit $Ax$ is bounded there is nothing to prove. According to
Birch and Swinnerton-Dyer \cite{BirchSD}, if $Ax$ is
unbounded 
then $\overline{Ax}$ contains a lattice with a
representative as in~\eqref{block form} (up to a possible permutation
of the coordinates) with $k=2$. Now the claim follows using 
induction and appealing to
Lemma~\ref{lem: compt red}. 
\end{proof}
\begin{proposition}\label{copt red prop}
It is enough to establish Theorem~\ref{thm: main} for
lattices having a bounded $A$-orbit.
\end{proposition}
\begin{proof}
Let $x\in \Xn$ be arbitrary. By Lemma~\ref{BSD}, $\overline{A x}$
contains a lattice $\br{g}$ with $g$ of upper triangular block form
(up to a possible permutation of the coordinates)
with diagonal blocks representing lattices with bounded orbits under
the corresponding diagonal groups. Assuming Theorem~\ref{thm: main}
for lattices having bounded orbits, and applying Lemma~\ref{lem: compt
  red} we may take $g$ whose diagonal blocks represent  
stable lattices. By Lemma~\ref{lem: block stable is stable}, $\br{g}$ is
stable as well.
\end{proof}

\subsection{Some technical preparations}
We now discuss the  subgroups of a lattice $x\in \Xn$ which
almost attain the minimum $\al(x)$ in~\eqref{alpha}.  
 \begin{definition}\label{bn}
Given a lattice $x\in \Xn$ and $\del>0$, let 
\begin{align*}
\on{Min}_{\del}(x)&\defi\set{\Lam \subset x:\av{\Lam}^{\frac{1}{r(\Lam)}}<(1+\del)\al(x)},\\
\tb{V}_{\del}(x)&\defi\on{span}\on{Min}_{\del}(x),\\
\dim_\del(x)&\defi\dim\tb{V}_{\del}(x).
\end{align*}
\end{definition}
We will need the following technical statement. 
\begin{lemma}\label{for the inradius}
For any $\rho>0$ 
there exists a
neighborhood of the identity $W\subset 
G$ with the  
following property. Suppose  $ 2\rho \leq \delta_0 \leq
d+1$ and suppose  $x\in \Xn$ 
such that 
$\dim_{\delta_0-\rho}(x)=\dim_{\delta_0+\rho}(x)$. 
Then for any $g\in W$ and any 
$\del\in \left(\del_0-\frac{\rho}{2},\del_0+\frac{\rho}{2} \right)$ we have 
\begin{equation}\label{eq 1806}
\tb{V}_{\del}(gx)=g\tb{V}_{\del_0}(x).
\end{equation}
In particular, there is $1 \leq k \leq n$ such that  
for any $g\in W$ and any $\del\in
\left(\del_0-\frac{\rho}{2},\del_0+\frac{\rho}{2} \right)$, 
$\dim_\del(gx)=k$. 
\end{lemma}
\begin{proof}
Let $c>1$ be chosen close enough to 1 so that for $2\rho \leq \delta_0
\leq d+1$ we have
\eq{eq: defn c}{c^2\left(1+\del_0+\frac{\rho}{2} \right) < 1+\del_0 +\rho \ \ \text{and
} \ \frac{1+\del_0-\frac{\rho}{2}}{c^2} > 
1+\del_0-\rho.}
Let $W$ be a small enough neighborhood of the identity
in $G$, so that for any discrete subgroup $\Lam \subset \bR^n$ we have 
\begin{equation}\label{22.2.2}
g\in W \ \ \implies \ \ c^{-1}\av{\Lam}^\frac{1}{r(\Lam)}\le
\av{g\Lam}^\frac{1}{r(g\Lam)}\le c\av{\Lam}^\frac{1}{r(\Lam)}. 
\end{equation}
Such a neighborhood exists since  the linear action of $G$
on $\bigoplus_{k=1}^n\bigwedge^k_1 \R^n$ is continuous, and since 
we can write $|\Lam| = \|v_1 \wedge \cdots \wedge v_r\|$
where $v_1, \ldots, v_r$ is a generating set for $\Lam$.
%
%
It follows from~\eqref{22.2.2} that for any $x\in \Xn$ and $g\in W$ we have 
\eq{22.2.3}{
c^{-1}\al(x)\le \al(gx)\le c\al(x).
}
 Let $\del\in \left(\del_0-\frac{\rho}{2}, \del_0+\frac{\rho}{2}
 \right)$ and $g\in W$. We will show below that 
 \begin{equation}\label{22.2.1'}
g\on{Min}_{\del_0-\rho}(x)\subset
\on{Min}_{\del}(gx)\subset
g\on{Min}_{\del_0+\rho}(x). 
\end{equation}
Note first that 
\equ{22.2.1'} implies the assertion of the Lemma; indeed, since
$\tb{V}_{\del_1}(x) \subset \tb{V}_{\del_2}(x)$ for $\delta_1 < \delta_2$, and since we assumed
that 
$\dim_{\delta_0-\rho}(x)=\dim_{\delta_0+\rho}(x)$, 
we see that 
$\tb{V}_{\del_0}(x)=\tb{V}_{\del}(x)$ for $\delta_0-\rho \leq \delta
\leq \delta_0+\rho$. So by \equ{eq: defn c}, the
subspaces spanned by the two sides of \eqref{22.2.1'}
are equal to $g\tb{V}_{\del_0}(x)$ and \eqref{eq 1806}
follows. 

It remains to prove~\eqref{22.2.1'}. Let
$\Lam\in\on{Min}_{\del_0-\rho}(x)$. Then we find 
\[
\begin{split}
\av{g\Lam}^{\frac{1}{r(g\Lam)}} & \stackrel{\eqref{22.2.2}}{\leq}
c\av{\Lam}^{\frac{1}{r(\Lam)}} \leq c(1+\delta_0 -\rho) \alpha(x) \\ &
\stackrel{\equ{eq: defn c}}{\le}
c^{-1}\left(1+\del_0-\frac{\rho}{2} \right)\al(x) \stackrel{\equ{22.2.3}}{<}(1+\del)\al(gx).
\end{split}\]
By definition this means that $g\Lam\in\on{Min}_{\del}(gx)$ which
establishes the first  inclusion in \eqref{22.2.1'}. The second
inclusion is similar and is left to the reader.  
\ignore{
For the second inclusion, let $\Lam \subset x$ such that
$g\Lam\in\on{Min}_{(\del)}(gx)$. Using the definition and~\eqref{22.2.2},\eqref{22.2.3} we conclude that 
\[\av{\Lam}^{\frac{1}{r(\Lam)}} \leq c\av{g\Lam}^{\frac{1}{r(g\Lam)}}\le
c(1+\del)\al(gx)< c^2 (1+\del_0+\frac{\rho}{2})\al(x).\]
I.e.\ $\Lam\in\on{Min}_{c^2(1+\del_0+\frac{\rho}{2})}(x)$ which
establishes the right inclusion in~\eqref{22.2.1'}. }
\end{proof}

\subsection{The cover of $A$} 
Let $x\in \Xn$ and let $\vre>0$ be given. Define
$\cU^{x,\vre}=\left\{U^{x,\vre}_i \right\}_{i=1}^n$ where 
\begin{equation}\label{the cover}
U^{x,\vre}_k\defi\set{a\in A: \on{dim}_\del(ax)=k\textrm{ for $\del$
    in a neighborhood of }k\vre}. 
\end{equation} 
\begin{theorem}\Name{order of cover}
Let $x\in \Xn$ be such that $Ax$ is bounded. Then for any $\vre \in (0,1)$, 
$U^{x,\vre}_n\neq \varnothing.$ 
\end{theorem}
In this subsection we will reduce the proof of Theorem \ref{thm: main}
to Theorem \ref{order of cover}. This will be done via the following
statement, 
which could be interpreted as saying that a 
lattice satisfying $ \dim_{\delta}(x)=n
$ is `almost stable'. 

\begin{lemma}\label{not growing lemma}
For each $n$, there exists a positive function $\psi(\del)$ with
$\psi(\del) \to_{\del\to 0}0$, such that for any $x\in \Xn$,
\eq{eq: lemma first part}{\set{\Lam_i}_{i=1}^\ell\subset\on{Min}_{\del}(x) \ \implies \ 
\Lam_1+\dots+\Lam_\ell\in\on{Min}_{\psi(\del)}(x).
}
In particular, if $\dim_\del(x)=n$ then $\al(x)\ge (1+\psi(\del))^{-1}$.  
\end{lemma}
\begin{proof}
Let $\Lam,\Lam'$ be two discrete subgroups of $\bR^d$.
The following inequality is straightforward to 
prove via the Gram-Schmidt procedure for computing $|\Lam|$: 
\begin{equation}\label{volume formula}
\av{\Lam+\Lam'}\le\frac{\av{\Lam}\cdot\av{\Lam'}}{\av{\Lam\cap\Lam'}}.
\end{equation}
Here we adopt the convention that  $\av{\Lam\cap\Lam'}=1$ when
$\Lam\cap\Lam'=\set{0}$. 
Let $x\in \Xn$ and let $\set{\Lam_i}_{i=1}^\ell\subset
\on{Min}_{\del}(x)$. Assume first that  
$\ell\le n$. We prove by induction on $\ell$ the existence of a
function  $\psi_\ell(\del)\overset{\del\to0}{\lra}0$  
for which
$\Lam_1+\dots+\Lam_\ell\in\on{Min}_{\psi_\ell(\del)}(x)$. For
$\ell=1$ one can trivially pick $\psi_1(\del)=\del$.  
Assuming the existence of $\psi_{\ell-1}$,  set $\Lam=\Lam_1$,
$\Lam'=\Lam_2+\dots+\Lam_\ell$, $\al=\al(x)$ and note  
that $r(\Lam+\Lam')=r(\Lam)+r(\Lam')-r(\Lam\cap\Lam')$. We deduce
from~\eqref{volume formula} and the definitions that 
\begin{align}\label{eneq2007}
\nonumber \av{\Lam+\Lam'}&\le
\frac{\av{\Lam}\cdot\av{\Lam'}}{\av{\Lam\cap\Lam'}}
\le
\frac{\pa{(1+\del)\al}^{r(\Lam)}\pa{(1+\psi_{\ell-1}(\del))\al}^{r(\Lam')}}{\al^{r(\Lam\cap\Lam')}}\\  
&=(1+\del)^{r(\Lam)}(1+\psi_{\ell-1}(\del))^{r(\Lam')}\al^{r(\Lam+\Lam')}.
\end{align}
Hence, if we set 
$$\psi_\ell(\del)\defi \max
\pa{(1+\del)^{r(\Lam)}(1+\psi_{\ell-1}(\del))^{r(\Lam')}}^{\frac{1}{r(\Lam+\Lam')}}-1,$$
where  
the maximum is taken over all possible values of $r(\Lam), r(\Lam'),
r(\Lam+\Lam')$ then $\psi_\ell(\del)\lra_{\del\to 0}0$  
and~\eqref{eneq2007} implies that 
$\Lam+\Lam'\in\on{Min}_{\psi_\ell(\del)}(x)$ as desired. We take
$\psi(\del)\defi\max_{\ell=1}^n\psi_\ell(\del).$ 
Now if $\ell >n$  one can find a subsequence 
$1\le i_1<i_2\dots<i_d\le n$  such that
$r(\sum_{i=1}^\ell\Lam_i)=r(\sum_{j=1}^d\Lam_{i_j})$ and in
particular,  
$\sum_{j=1}^d\Lam_{i_j}$ is of finite index in $\sum_{i=1}^\ell\Lam_i$. From the first part of the 
argument we see that $\sum_{j=1}^d\Lam_{i_j}\in \on{Min}_{\psi(\del)}(x)$ and as the covolume of 
$\sum_{i=1}^\ell\Lam_i$ is not larger than that of $\sum_{j=1}^d\Lam_{i_j}$ we deduce that 
$\sum_{i=1}^\ell\Lam_i \in\on{Min}_{\psi_\ell(\del)}(x)$ as well.

To verify the last assertion, note that
when $\dim_\del(x)=n$, \equ{eq: lemma first part} implies the
existence of a finite index subgroup $x'$ 
of $x$ belonging to $\on{Min}_{\psi(\del)}(x)$. In particular, 
$1\leq \av{x'}^{\frac{1}{n}}\le(1+\psi(\del))\al(x)$ as desired.
\end{proof}

\ignore{
The following statement essentially 
says that a lattice $x$ satisfying $\dim_\del(x)=n$ with $\del$ small
can be considered to be `almost stable'.
\begin{corollary}\label{hitting stable cor}
If $x_j\in \Xn$ satisfies $\dim_{\del_j}(x_j)=n$ for some sequence
$\del_j \to 0$, then any accumulation point of 
$\set{x_j}$ is a stable lattice.
\end{corollary}
\begin{proof}
By Lemma~\ref{not growing lemma} we have  
$$1\ge
\limsup\al(x_j)\ge \liminf \al(x_j)\ge \lim (1+\psi(\del_j))^{-1}=1,$$ 
which shows that $\lim\al(x_j)=1$.  
The function $\al$ is continuous on $\Xn$ and therefore if $x$  is an
accumulation point of $\set{x_j}$ then $\al(x)=1$, i.e. $x$ is stable.
\end{proof}
}
\begin{proof}[Proof of Theorem~\ref{thm: main} assuming
  Theorem~\ref{order of cover}] 
By Proposition~\ref{copt red prop} we may assume that $Ax$ is
bounded. Let $\vre_j \in (0,1)$ so that $\vre_j \to_j 0$. By
Theorem~\ref{order of cover} we know that
$U^{x,\vre_j}_n 
\neq \varnothing$. This means there is a sequence $a_j\in A$ such that 
$\dim_{\del_j}(a_jx)=n$
where $\del_j=n\vre_j\to 0$. 
The sequence $\set{a_jx}$ is
  bounded, and hence has limit points, so passing to a subsequence we
  let $x' \df \lim a_jx.$ 
By Lemma~\ref{not growing lemma} we have  
$$1\ge
\limsup_j\al(a_jx)\ge \liminf_j \al(a_j x)\ge \lim_j (1+\psi(\del_j))^{-1}=1,$$ 
which shows that $\lim_j\al(a_j x)=1$.  
The function $\al$ is continuous on $\Xn$ and therefore $\al(x')=1$,
i.e. $x' \in \overline{Ax}$ is stable.
\end{proof}
\section{Covers of Euclidean space}\Name{establishing
  topological input}  
In this section we will prove Theorem \ref{order of cover}, thus
completing the proof of Theorem \ref{thm: main}. Our main
tool will be McMullen's 
Theorem~\ref{topological input}. Before stating it we introduce some
terminology. We fix an invariant metric on $A$, and let $R>0$ and $k \in \{0, \ldots, n-1\}$. 
\begin{definition}\Name{def: almost affine}
We say that a subset $U\subset A$ is $(R,k)$-\textit{almost affine} if it is
contained in an $R$-neighborhood of a coset of a connected $k$-dimensional 
subgroup of $A$. 
\end{definition}
\begin{definition}\Name{def: inradius}
An open cover $\cU$ of $A$ is said to have \textit{inradius} $r>0$ if
for any $a\in A$ there exists $U\in\cU$ such that $B_r(a)\subset
U$, where $B_r(a)$ denotes the ball in $A$ of radius $r$ around $a$. 
\end{definition}
\begin{theorem}[Theorem 5.1 of~\cite{McMullenMinkowski}]\Name{topological input}
Let $\cU$ be an open cover of $A$ with inradius $r>0$ and let
$R>0$. Suppose that for any $1\le k\le n-1$, 
every connected component $V$ of the intersection of  
$k$ distinct elements of $\cU$ is
$(R,(n-1-k))$-almost affine. Then there is 
a point in $A$ which belongs to  
at least $n$ distinct elements of $\cU$. In particular, there are at
least $n$ distinct non-empty sets in $\cU$. 
\end{theorem}

The hypotheses of McMullen's theorem were slightly weaker but the
version above is sufficient for our purposes. 
We give a different proof of 
Theorem \ref{topological input} in this paper; namely it follows from
the more general Theorem
\ref{thm: covering}, which is proved
in Appendix \ref{appendix: Levin}. 

\ignore{
\begin{proposition}\label{assumptions hold}
If $x\in \Xn$ has a bounded $A$-orbit and $\vre>0$ then the collection
$\cU^{x,\vre}$ is an open cover of $A$ with positive inradius  
such that at least one of the following two possibilities hold:
\begin{enumerate}
\item $U^{x,\vre}_n\ne\varnothing$.
\item The hypothesis of Theorem~\ref{topological input} are satisfied.
\end{enumerate}
\end{proposition}
\begin{proof}[Proof of Theorem~\ref{order of cover} assuming
  Proposition~\ref{assumptions hold}] 
By the Proposition the possibility that $U^{x,\vre}_d=\varnothing$ is
ruled out as if this is the case then we may apply  
Theorem~\ref{topological input} and deduce that $\cU^{x,\vre}$ must
contain at least $n$ non-empty sets and in particular,  
$U_n^{x,\vre}\ne\varnothing$ which contradicts our assumption.
\end{proof} 
}

\subsection{Verifying the hypotheses of Theorem \ref{topological
    input} }
Below we fix a compact set $K\subset \Xn$ and a lattice $x$ for which
$Ax\subset K$. Furthermore, we fix $\vre>0$ and denote  
the collection $\cU^{x, \vre}$ defined in~\eqref{the cover} by
$\cU=\set{U_i}_{i=1}^n$.
\begin{lemma}\Name{lem: positive inradius}
The collection $\cU$ forms an open cover of $A$ with positive inradius.
\end{lemma}
\begin{proof}
The fact that the sets $U_i\subset A$ are open follows readily from
the requirement in~\eqref{the cover} that $\on{dim}_\del$ is constant
in a neighborhood of $\del=k\vre$.  
Given $a\in A$, let $1\le k_0\le n$ be the minimal number $k$ for which
$\dim_{(k+\frac{1}{2})\vre}(ax)\le k$  
(this inequality holds trivially for $k=n$). 
From the minimality of $k_0$ we conclude that 
$\dim_\del(ax)=k_0$ for
any  
$\del\in \left[\pa{k_0-\frac{1}{2}}\vre,\pa{k_0+\frac{1}{2}}\vre \right]$. 
This
shows that $a\in U_{k_0}$ so 
$\cU$ is indeed a cover of $A$. 

We now show that the cover has positive inradius. 
Let $W \subset G$ be the open neighborhood of the identity obtained
from 
Lemma~\ref{for the inradius} for  $\rho \df \frac{\vre}{2}$.
Taking $\delta_0 \df k_0 \vre$ we find that 
for any 
$g\in W$,   
$\del\in \pa{\pa{k_0-\frac{1}{4}}\vre,\pa{k_0+\frac{1}{4}}\vre}$ we
have that $\dim_\del(gax)=k_0$. This shows that  
$(W\cap A)a\subset U_{k_0}$. Since $W\cap A$ is an open neighborhood
of the identity in $A$ and the metric on $A$ is invariant under  
translation by elements of $A$, there exists $r>0$ 
(independent of $k_0$ and $a$) so that $B_r(a)\subset U_{k_0}$. In
other words, the inradius of $\cU$ is positive as desired. 
\end{proof}

The following will be used for verifying the second hypothesis of
Theorem~\ref{topological input}.
\begin{lemma}\Name{flat things}
There exists $R>0$ such that any connected component of $U_k$ is $(R,k-1)$-almost affine.
\end{lemma}
\begin{definition}\label{cv}
For a discrete subgroup $\Lam \subset \bR^d$ of rank $k$, 
let $$c(\Lam)\defi\inf\set{\av{a\Lam}^{1/k}:a\in A},$$ and say that
$\Lam$ is {\it incompressible} if $c(\Lam)>0$. 
\end{definition}
Lemma~\ref{flat things} follows from:
\begin{theorem}[{\cite[Theorem 6.1]{McMullenMinkowski}}]\Name{finite
    distance from a group} 
For any positive $c,C$ there exists $R>0$ such that if 
$\Lam \subset \bR^n$ is an incompressible discrete subgroup of rank
$k$ with $c(\Lam)\ge c$ then  
$\set{a\in A: \av{a\Lam}^{1/k}\le C}$ is $(R,j)$-almost affine for some $j\le
\gcd(k,n)-1$. 
\end{theorem}
\begin{proof}[Proof of Lemma~\ref{flat things}]
We first claim that there exists $c>0$ such that  
for any discrete subgroup $\Lam \subset x$ we have that $c(\Lam)\ge
c$. To see this, recall that $Ax$ is contained in a compact subset
$K$, and hence by Mahler's compactness criterion,  there is a positive
lower bound on 
the length of any non-zero vector
belonging to a lattice in $K$. 
On the other
hand, Minkowski's convex body theorem shows that the shortest nonzero
vector in a discrete subgroup $\Lambda \subset \R^n$ is bounded above
by a constant multiple of $|\Lam|^{1/r(\Lam)}$. This implies the
claim.

In light of Theorem \ref{finite
    distance from a group}, it suffices to show that there is $C>0$
such that  if $V\subset U_k$ is a
connected component, then there exists 
$\Lam \subset x$ such that  $V\subset\set{a\in A: \av{a\Lam}^{1/k}\le C}$.
For any $1\le k\le n$, write $\tb{gr}_k$ for the Grassmannian of
$k$-dimensional subspaces of $\bR^n$.  
Define
$$
\cM:U_k\to \tb{gr}_k, \ \ 
\cM(a)\defi a^{-1}\tb{V}_{k\vre}(ax).$$ 
Observe that $\cM$ is locally constant on $U_k$. Indeed, by
definition of $U_k$, 
for $a_0\in U_k$ there exists 
$0<\rho< \frac{\vre}{2}$ such that 
$\dim_\del(a_0x)=k$ for any $\del\in(k\vre-\rho,k\vre+\rho)$. Applying
Lemma~\ref{for the inradius}  
for the lattice $a_0x$ with $\rho$ and $\del_0=k\vre$
we see 
that for any $a$ in a neighborhood
of the identity in $A$, 
\begin{align*}
\cM(aa_0)&=a_0^{-1}a^{-1}\tb{V}_{k\vre}(aa_0x)
=a_0^{-1}\tb{V}_{k\vre}(a_0x)=\cM(a_0).
\end{align*}

Now let $\Lam \df x\cap \cM(a)$ where $a\in V$; $\Lam$ is well-defined
since $\cM$ is locally
constant. 
Then 
for $a \in V$, 
\begin{align*}
a\Lam&=a(x\cap\cM(a)) =a(x\cap a^{-1}\tb{V}_{k\vre}(ax))=ax\cap\tb{V}_{k\vre}(ax).
\end{align*}
By Lemma~\ref{not growing lemma} we have that 
\begin{align*}
\av{a\Lam}^{1/k}=\av{
  ax\cap \tb{V}_{k\vre}(ax)}^{1/k}<(1+\psi(k\vre))\al(ax). 
\end{align*}
Since $\alpha(ax) \leq 1$ we may take $C \df 1+\psi(k\vre)$ to
complete the proof. 
%
\end{proof}

\begin{proof}[Proof of Theorem \ref{order of cover}]
Assume by contradiction that $Ax$ is bounded but $U_n^{x, \vre} =
\varnothing$ for some $\vre \in (0,1)$. Then by Lemma \ref{lem: positive inradius}, 
$$\cU \df \left \{U_1,
\ldots, U_{n-1} \right \}, \text{ where } U_j \df U_j^{x, \vre}, $$  
is a cover of $A$ of positive inradius. Moreover, if $V$ is a
connected component of $U_{j_1} \cap \cdots \cap U_{j_k}$ with $j_1 < \cdots < j_k \leq n-1$, 
then $V_k \subset U_{j_1}$ and $j_1 \leq n-k$. So in
light of Lemma \ref{flat things}, 
the hypotheses of Theorem~\ref{topological input} are satisfied.
We 
deduce that $\cU = \left\{U_1, \ldots, U_{n-1} \right \}$
contains at least $n$ elements, which is impossible. 
\end{proof}

\section{Bounds on Mordell's constant}\Name{sec: rankin bounds} 
In analogy with~\eqref{alpha} we define for
any $x\in \Xn$ and $1\le k\le n$,  
\begin{align}\Name{eq: k quantities}
\cV_k(x)&\defi\set{\av{\Lam}^{1/r(\Lam)}:\Lam \subset x, r(\Lam)=k},\\ 
\al_k(x)&\defi\min\cV_k(x). 
\end{align}

The following is clearly a consequence of Theorem \ref{thm: main}:
\begin{cor}\Name{cor: Euclidean}
For any $x \in \Xn$, any $\vre>0$  and any $k \in \{1, \ldots, n\}$ there is $a \in
A$ such that $\alpha_k(ax) \geq 1-\vre$. 
\end{cor}
As the lattice $x = \Z^n$ shows, the constant 1 appearing in
this corollary cannot be improved for any $k$. Note also that the case
$k=1$ of 
Corollary \ref{cor: Euclidean}, although not stated explicitly in
\cite{McMullenMinkowski}, could be derived easily from McMullen's results in
conjunction with \cite{BirchSD}. 

\begin{proof}[Proof of Corollary \ref{cor: Ramharter conj}]
 Since the $A$-action maps a symmetric box $\mathcal{B}$ to a
 symmetric box of the same volume, the function $\kappa : \Xn \to \R$
 in \equ{eq: defn const} is $A$-invariant. By the case $k=1$ of
 Corollary \ref{cor: Euclidean}, for any $\vre>0$ and any $x \in \Xn$
 there is $a \in A$ such that $ax$ does not contain nonzero vectors of
 Euclidean length at most $1-\vre$, and hence does not contain nonzero vectors
 in the cube $\left [-\left(\frac{1}{\sqrt{n}} - \vre\right),
 \left(\frac{1}{\sqrt{n}} - \vre\right) \right ]^n$. This implies that
$\kappa(x) \geq \left(\frac{1}{\sqrt{n}} \right)^n$, as claimed. 
\end{proof}

We do not know whether the bound $\kappa_n \geq n^{-n/2}$ is
asymptotically optimal. However, it is not optimal for any fixed
dimension $n$:
\begin{proposition} \Name{prop: not optimal}
For any $n$, $\kappa_n > n^{-n/2}$. 
\end{proposition}
\begin{proof}
It is clear from the definition of the functions $\kappa$ and
$\alpha_k$ that if $x_j \to x_0$ in $\Xn$, then 
$$\kappa(x_0) \leq \liminf_j \kappa(x_j) \ \ \text{and } \
\alpha_k(x_0) \geq \limsup_j \alpha_k(x_j).$$
A simple compactness argument implies that the
infimum in \equ{eq: defn kappan} is attained, that is there is 
$x \in \Xn$ such that $\kappa_n = \kappa(x)$; moreover, for any $x_0
\in \overline{Ax}, \kappa(x_0) = \kappa(x)=\kappa_n$. Using the case $k=1$ of Corollary
\ref{cor: Euclidean}, we let $x_0$ be a stable lattice in
$\overline{Ax}$ such that $\alpha_1(x_0)
\geq 1$.  That is, $x_0$ contains no vectors in the open unit
Euclidean ball, so the open cube 
$C \df \left( -\frac{1}{\sqrt{n}},   \frac{1}{\sqrt{n}}\right)^n$ is
admissible. Moreover, the only possible vectors in $x_0$ on $\partial
\, C$ are on the corners of $C$, so there is $\vre>0$ such that the box
$C' \df \left( -\frac{1}{\sqrt{n}},   \frac{1}{\sqrt{n}}\right)^{n-1} \times
\left(-\left(\frac{1}{\sqrt{n}} + \vre \right) ,
  \frac{1}{\sqrt{n}}+ \vre\right)$ is also admissible. Taking closed
boxes $\cB \subset C'$ with volume arbitrarily close to that of $C'$,
we see that 
$$
\kappa_n = \kappa(x_0) \geq \frac{\Vol(C')}{2^n} > n^{-n/2}.
$$ 
\end{proof}

\ignore{
We take this opportunity to mention another connection between the
Mordell constant and the dynamics of the $A$-action on
$\Xn$. 

\begin{proposition}\Name{prop: conjecture minimizers}
There is $x \in \Xn$ with a bounded n $\kappa_n$ is attained on a compact $A$-orbit. 
\end{proposition} 
\begin{proof}

\end{proof} 

The following is a well-known conjecture:
$$
\text{(CSDM)} \ \ \text{any bounded } A \text{-orbit on } \Xn \text{ is compact.} 
$$
This first appeared in the paper \cite{} of 
 of Cassels and
Swinnerton-Dyer, and was recast in dynamical terms by Margulis in
\cite{}. Moreover compact
$A$-orbits correspond to algebraic lattices obtained from orders in
totally real number fields, see \cite{LW}.
}
Our next goal is Corollary \ref{cor: 1 mod 4} which gives an explicit
lower bound on $\kappa_n$, which 
improves \equ{eq: our bound} for $n$ congruent to 1 mod 4. To obtain our bound 
we treat separately lattices with bounded or unbounded
$A$-orbits. If $Ax$ is unbounded we bound $\kappa(x)$ by using an inductive
procedure and the work of Birch and Swinnerton-Dyer, as in \S
\ref{sec: reduction to compact orbits}.  In the bounded case we use arguments of
McMullen and known 
upper bounds for Hadamard's determinant problem. Our method
applies with minor modifications whenever $n$ is not divisible by
4. 
We begin with an analogue of Lemma \ref{lem: block stable is stable}. 

\begin{lemma}\Name{lem: bound on kappa in blocks}
Suppose $x = [g] \in \Xn$ with $g$ in upper triangular block form as
in \equ{block form}. Then $\kappa(x) \geq \prod_1^k \kappa
([g_i])$. In particular $\kappa(x) \geq \prod_1^k n_i^{-n_i/2}.$ 
\end{lemma}
\begin{proof}
By induction, it suffices to prove the Lemma in case $k=2$. In this
case there is a direct sum decomposition $\R^n = V_1 \oplus V_2$ where
the $V_i$ are spanned by standard basis vectors, and if we write $\pi:
\R^n \to V_2$ for the corresponding projection, then $[g_1] = x \cap
V_1, [g_2] = \pi(x)$. Write $\kappa^{(i)} \df \kappa([g_i])$. Then for
$\vre>0$, 
there are symmetric boxes $\mathcal{B}_i \subset V_i$ 
such that $\mathcal{B}_i$ is admissible for $[g_i]$ and 
$$\Vol(\mathcal{B}_i) \geq \frac{\kappa^{(i)} - \vre}{2^{n_i}}.$$
We claim that $\mathcal{B} \df \mathcal{B}_1 \times \mathcal{B}_2$ is
admissible for $x$. To see this, suppose $u \in x \cap
\mathcal{B}$. Since $\pi(u) \in \mathcal{B}_2$ and $\mathcal{B}_2$ 
is admissible for $\pi(x) = [g_2]$ we must have
$\pi(u) =0$, i.e. $u \in x \cap V_1 = [g_1]$; since
$\mathcal{B}_1$ is admissible for $[g_1]$ we must have $u=0$. 

This implies
$$\kappa(x) \geq 2^n \Vol(\mathcal{B})  = 2^{n_1} \Vol(\mathcal{B}_1)
\cdot 2^{n_2} 
\Vol(\mathcal{B}_2) \geq (\kappa^{(1)} -
  \vre)(\kappa^{(2)}-\vre),$$ 
and the result follows taking $\vre \to 0$. 
\end{proof}

\begin{corollary}\Name{cor: kappa unbounded orbits}
If $x \in \Xn$ is such that $Ax$ is unbounded then 
\eq{eq: kappa star}{
\kappa(x) \geq 
(n-1)^{-(n-1)/2}.
}
\end{corollary}
\begin{proof}
If $Ax$ is unbounded then by \cite{BirchSD}, up to a permutation of
the axes, there is $x' \in
\overline{Ax}$ so that $x' = [g]$ is in upper triangular form, with $k
\geq 2$ blocks. Let the
corresponding parameters as in \equ{block form} be $n= n_1+ \cdots
+n_k$. Since $\kappa(x)
\geq \kappa(x')$, by Lemma \ref{lem: bound on kappa in blocks} it
suffices to prove 
that 
\eq{eq: suffices to prove that}{
\prod_{i=1}^k \frac{1}{n_i^{n_i/2}} \geq \frac{1}{(n-1)^{\frac{n-1}{2}}}.
}
It is easy to check that for $j=1, \ldots, n-1$, 
$$
j^{\frac{j}{2}}(n-j)^{\frac{n-j}{2}} \leq (n-1)^{\frac{n-1}{2}},
$$ and the case $k=2$ of \equ{eq: suffices to prove that} follows. 
By induction on $k$ one then shows that 
$
\prod_{i=1}^k n_i^{-n_i/2} \geq (n-k+1)^{-\frac{n-k+1}{2}}
$
and this implies \equ{eq: suffices to prove that} for all $k\geq 2$. 
\end{proof}
To treat the bounded orbits we will use known bounds on the Hadamard
determinant problem, which we now
recall. Let 
\eq{eq: defn hn}{
h_n \df \sup \left \{ |\det (a_{ij})|: \forall i,j \in \{1, \ldots, n\},
|a_{ij}| \leq 1 \right \}.
}
Hadamard showed that $h_n \leq n^{n/2}$ and proved that this bound is
not optimal unless $n$ is equal to 1,2 or is a multiple of 4. Explicit 
upper bounds for  such $ n $ have been obtained
by Barba, Ehlich and Wojtas (see \cite{brenner, wiki_hadamard}). 

\begin{proposition}\Name{prop: improving using Hadamard}
If $x \in \Xn$ has a bounded $A$-orbit then $\kappa(x) \geq
\frac{1}{h_n}$. 
\end{proposition}

\begin{proof}[Sketch of proof]
Let $\vre>0$. There is $p <\infty$ such that the $L^p$ norm and the
$L^\infty$ norm on $\R^n$ are $1+\vre$-biLipschitz; i.e. for any $v
\in \R^n$, 
\eq{eq: bilipschitz}{
\frac{\|v\|_p}{1+\vre} \leq \|v\|_{\infty} \leq (1+\vre) \|v\|_p.
}
In \cite{McMullenMinkowski}, McMullen showed that the closure of any
bounded $A$-orbit contains a well-rounded lattice, i.e. a lattice
whose shortest nonzero vectors span $\R^n$. In McMullen's paper, the
length of the shortest vectors was measured using the Euclidean
norm, but {\em McMullen's arguments apply equally well to the shortest
  vectors with respect to the $L^p$ norm}. Thus there is $a \in A$ and
vectors $v_1, \ldots, v_n \in ax$ spanning $\R^n$,  such that  for $i=1, \ldots, n$, 
$$ \|v_i\| \in [r, (1+\vre)r]. 
$$
Here $r$ is the length, with respect to
the $L^p$-norm, of the shortest nonzero vector of $ax$. Using the two
sides of \equ{eq: bilipschitz} we find that $ax$
contains an admissible symmetric box of sidelength $r/(1+\vre)$, and
the $L^\infty$ norm of the $v_i$ is at most $(1+\vre)^2 r$. Let $A$ be
the matrix whose columns are the $v_i$. Since the $v_i$ span $\R^n$,
$\det A \neq 0$, and since $x$ is unimodular, $|\det A| \geq
1$. Recalling \equ{eq: defn hn} we find that
$$1 \leq |\det A| \leq \left((1+\vre)^2)r\right)^{n} h_n,
$$
and by definition of $\kappa$ we find
$$
\kappa(x) = \kappa(ax) \geq \left(\frac{r}{1+\vre} \right)^n.
$$
Putting these together and letting $\vre \to 0$ we see that 
$\kappa(x) \geq \frac{1}{h_n}$, as claimed. 
\end{proof}
\begin{corollary}\Name{cor: 1 mod 4}
If $n \geq 5$ is congruent to 1 mod 4, then 
\eq{eq: better bound}{
\kappa_n \geq \frac{1}{\sqrt{2n-1}(n-1)^{(n-1)/2}}.
}
\end{corollary}
\begin{proof}
The right hand side of
\equ{eq: better bound} is clearly smaller than the right hand side of
\equ{eq: kappa star}. Now the
claim follows from Corollary \ref{cor: kappa unbounded orbits} and
Proposition \ref{prop: improving using Hadamard}, using Barba's bound
\eq{eq: Barba}{
h_n \leq \sqrt{2n-1}(n-1)^{(n-1)/2}.
} 
\end{proof}
The same argument applies in the other cases in which $n$ is
sufficiently large and is not divisible by 4, since in these cases
there are explicit upper bounds for the numbers $h_n$ which could be
used in place \equ{eq: Barba}. 

\section{Two strategies for Minkowski's conjecture}\Name{sec: MC} 
We begin by recalling the well-known Davenport-Remak strategy
for proving Minkowski's conjecture. The function 
$N(u) = \prod_1^n u_i$ is clearly $A$-invariant, and 
it follows that the quantity 
$$\widetilde{ N}(x) \df  \sup_{u \in \R^n} \inf_{v
  \in x} |N(u-v)|$$ 
appearing in \equ{eq: Minkowski conj} is
$A$-invariant. Moreover, it is easy to show that if $x_n \to x$ in
$\Xn$ then $\widetilde{ N}(x) \geq \limsup_n \widetilde{ N}(x_n)$. Therefore, in order
to show the estimate \equ{eq: Minkowski conj} for $x' \in \Xn$, it is
enough to show it for some $x \in \overline{Ax'}$. Suppose that $x$
satisfies \equ{eq: covrad} with $d=n$; that is for every $u \in \R^n$ there is $v
\in x$ such that $\|u-v\| \leq \frac{\sqrt{n}}{2}$. 
Then applying  
the inequality of arithmetic and geometric means one finds
$$\prod_1^n \left(|u_i-v_i|^2 \right)^{\frac1n} \leq \frac{1}{n} \sum_1^n |u_i-v_i|^2 \leq \frac{1}{4}
$$
which implies $|N(u-v)| \leq \frac{1}{2^n}$. 
The upshot is that in order to prove Minkowski's conjecture, it is
enough to prove that for every $x' \in \Xn$ there is $x \in
\overline{Ax}$ satisfying \equ{eq: covrad}. So in light of Theorem
\ref{thm: main} we obtain:
\begin{corollary}\Name{cor: for Minkowski 1}
If all stable lattices in $\Xn$ satisfy \equ{eq: covrad}, then
Minkowski's conjecture is true in dimension $n$. 
\end{corollary}

In the next two subsections, we outline two strategies for
establishing that all stable lattices satisfy
\equ{eq: covrad}. 
Both strategies yield affirmative answers in dimensions $n
\leq 7$, thus providing new proofs of Minkowski's conjecture in these
dimensions. 

\subsection{Using Korkine-Zolotarev reduction}
Korkine-Zolotarev reduction is a classical method for
choosing a basis $v_1, \ldots, v_n$ of a lattice $x \in \Xn$. Namely
one takes for $v_1$ a shortest nonzero vector of $x$ and
denotes its length by $A_1$. Then, proceeding inductively, for $v_i$ one takes 
a vector whose projection onto $(\spa(v_1, \ldots, v_{i-1}))^\perp$ is
shortest (among those with nonzero projection), and denotes the length
of this
projection by $A_i$. In case there is more
than one shortest vector the process 
is not uniquely defined. Nevertheless we call $A_1, \ldots, A_n$ the
{\em diagonal KZ coefficients of $x$} (with the understanding that
these may be multiply defined for some measure zero subset of $\Xn$). Since $x$
is unimodular we always have 
\eq{eq: det one}{\prod A_i 
=1.}
Korkine and Zolotarev proved the bounds 
\eq{eq: KZ bounds}{
A_{i+1}^2 \geq \frac34 A_i^2, \ \ A_{i+2}^2 \geq \frac23 A_i^2.
}

A method introduced by Woods \cite{Woods_n=4} and developed further in
\cite{hans-gill1} leads to an upper bound on
$\covrad(x)$ in terms of the diagonal KZ coefficients. 
The method relies on the following estimate. Below
$\gamma_n \df \sup_{x \in \Xn} \alpha_1(x)$, where $\alpha_1$
is defined via \equ{eq: k quantities}, that is, $\gamma^2_n$ is the
so-called Hermite constant. 

\begin{lemma}[Woods]\Name{lemma of Woods}
Suppose that $x$ is a lattice in $\R^n$ of covolume $d$, and suppose
that $2 A_1^n \geq
d \gamma_{n+1}^{n+1}$. Then 
$$
\covrad^2(x) \leq A_1^2 -\frac{A_1^{2n+2}}{d^2 \gamma_{n+1}^{2n+2}}.
$$
\end{lemma} 
Woods also used the following observation:
\begin{lemma}\Name{first Woods lemma}
Let $x$ be a lattice in $\R^n$, let $\Lambda$ be a subgroup, and let
$\Lambda'$ denote the projection of $x$ onto $(\spa 
\Lambda)^{\perp}$. Then 
$$
\covrad^2(x) \leq \covrad^2(\Lambda) +\covrad^2(\Lambda')
$$
\end{lemma} 
As a consequence of Lemmas \ref{lemma of Woods} and \ref{first Woods
  lemma},  we obtain:
\begin{proposition}\Name{prop: combining Woods}
Suppose $A_1, \ldots, A_n$ are diagonal KZ coefficients of $x \in
\Xn$ and suppose $n_1, \dots, n_k$ are positive integers with $n = n_1 + \cdots
+ n_k$. 
Set
\eq{eq: defn mi di}{m_i \df n_1 + \cdots
+ n_i \ \text{and } d_i \df \prod_{j=m_{i-1}+1}^{m_{i}} A_j.}
If 
\eq{eq: assumption of woods}{
2A_{m_{i-1}+1} \geq d_i \gamma_{n_i+1}^{n_i+1} 
}
for each $i$, then 
\eq{eq: consequence}{
\covrad^2(x ) \leq \sum_{i=1}^k \left(A^2_{m_{i-1}+1} -
  \frac{A^{2n_i+2}_{m_{i-1}+1}}{d_i^2 \gamma_{n_i+1}^{n_i+1}} \right) 
}

\end{proposition}
\begin{proof}
Let $v_1, \ldots, v_n$ be the basis of $x$
obtained by the Korkine Zolotarev reduction process. Let 
$\Lambda_1$ be the subgroup of $x$ generated by $v_1, \ldots,
v_{n_1}$, and for $i=2, \ldots, k$ let
$\Lambda_i$ be the projection onto
$(\bigoplus_1^{i-1} \Lambda_j)^{\perp}$ of the subgroup of $x$
generated by $v_{m_{i-1}+1}, \ldots, v_{m_{i}}$. This is a lattice of
dimension $m_i$, and arguing as in the proof of \equ{eq: claim 1} we
see that it has covolume $d_i$. The assumption
\equ{eq: assumption of woods} says that we may apply Lemma \ref{lemma
  of Woods}
to each $\Lambda_i$.  We obtain 
$$\covrad^2(\Lambda_i) \leq
A^2_{m_{i-1}+1} - \frac{A^{2n_i+2}_{m_{i-1}+1}}{d_i^2
  \gamma_{n_i+1}^{n_i+1}}  
$$ for each $i$,  and we combine these estimates using Lemma
\ref{first Woods lemma} and an obvious induction. 
\end{proof}

\begin{remark}
Note that it is an open question to determine the numbers $\gamma_n$;
however, if we have a bound $\tilde{\gamma}_n \geq \gamma_n$ we may
substitute it into Proposition \ref{prop: combining Woods} in place of $\gamma_n$, as this
only makes the requirement \equ{eq: assumption of woods} stricter and
the conclusion \equ{eq: consequence}
weaker.  
\end{remark}

Our goal is to apply this method to the problem of bounding the covering
radius of stable lattices. We note:
\begin{proposition}\Name{prop: KZ of stable}
If $x$ is stable then we have the inequalities 
\eq{eq: defn KZ stable}{
A_1 \geq 1, \ \ A_1 A_2 \geq 1, \ \ \ldots \ \ A_1 \cdots A_{n-1} \geq 1.
}
\end{proposition}
\begin{proof}
In the above terms, the number $A_1 \cdots A_i$ is equal to
$|\Lambda|$ where $\Lambda$ is the subgroup of $x$ generated by $v_1,
\ldots, v_i$. 
\end{proof}
This motivates the following:
\begin{definition}
We say that an $n$-tuple of positive real numbers $A_1, \ldots, A_n$
is {\em KZ stable} if the inequalities \equ{eq: det one}, \equ{eq: KZ
  bounds}, \equ{eq: defn KZ stable} are 
satisfied. We denote the set of KZ stable $n$-tuples by
$\mathrm{KZS}$. 
\end{definition}
Note that $\mathrm{KZS}$ is a compact subset of
$\R^n$. Recall that a {\em composition of $n$} is an ordered $k$-tuple $(n_1,
\ldots, n_k)$ of positive integers, such that $n=n_1+\ldots +n_k$. As
an immediate application of Corollary \ref{cor: for 
  Minkowski 1} and
Propositions \ref{prop: combining Woods} and \ref{prop: KZ of stable}
we obtain:
\begin{theorem}\Name{thm: use of KZ diagonal}
For each composition $\cI \df (n_1, \ldots, n_k)$ of $n$, define $m_i, d_i$ by
\equ{eq: defn mi di} and let $\mathcal{W}(\cI)$ denote
the set 
$$\left\{(A_1, \ldots, A_n): \forall
  i, \, 
\equ{eq: assumption of woods} \text{ holds, and } \sum_{i=1}^k \left(A^2_{m_{i-1}+1} -
  \frac{A^{2n_i+2}_{m_{i-1}+1}}{d_i^2 \gamma_{n_i+1}^{n_i+1}} \right) \leq
  \frac{n}{4} \right \}.
$$
If 
\eq{eq: covering suffices}{
\mathrm{KZS} \subset \bigcup_{\cI} \cW(\cI)
}
then Minkowski's conjecture holds in
dimension $n$. 
\end{theorem}
Rajinder Hans-Gill has informed the authors  that using
arguments as in \cite{hans-gill1, hans-gill2}, it is possible to
verify \equ{eq: covering suffices} 
in dimensions up to 7, thus
reproving Minkowski's conjecture in these dimensions. 

\subsection{Local maxima of covrad}
The aim of this subsection is to prove Corollary \ref{cor: Minkowski}, 
which shows that in
order to establish that all stable lattices in $\R^n$ satisfy the covering
radius bound \equ{eq: covrad}, it suffices to check this on a finite
list of lattices in each dimension $d \leq n$.


The function $\covrad : \Xn \to \R$ may have local maxima, in the
usual sense; that is, lattices $x \in \Xn$ for which there is a
neighborhood $\mathcal{U}$ of $x$ in $\Xn$ 
such that for all $x' \in \cU$ we have $\covrad(x') \leq
\covrad(x)$. Dutour-Sikiri\'c, Sch\"urmann and Vallentin
\cite{mathieu} gave a geometric characterization of lattices
which are local maxima of the function 
$\covrad$, and showed that there are finitely many in each dimension.
Corollary \ref{cor: Minkowski} asserts that Minkowski's conjecture
would follow if all local maxima of covrad satisfy the bound \equ{eq:
  covrad}.   
\begin{proof}[Proof of Corollary \ref{cor: Minkowski}]
We prove by induction on $n$ that any stable lattice satisfies the
bound \equ{eq: covrad} and apply Corollary \ref{cor: for Minkowski 1}. 
Let $\cS$ denote the set of stable lattices in $\Xn$. It is compact
so the function $\covrad$ attains a maximum on $\cS$, and it suffices
to show that this maximum is at most $\frac{\sqrt{n}}{2}$. Let $x \in \cS$
be a point at which the maximum is attained. If $x$ is an interior
point of $\cS$ then necessarily $x$ is a
local maximum for $\covrad$ and the required bound holds by
hypothesis. Otherwise, there is a sequence $x_j \to x$ such that
$x_j \in \Xn \sm \cS$; thus each $x_j$ contains a discrete subgroup
$\Lambda_j$ with $|\Lambda_j| <1$ and $r(\Lambda_j) <n$. Passing to a subsequence we may
assume that that $r(\Lambda_j)=k<n$ is the same for all $j$, and
$\Lambda_j$ converges to a discrete subgroup $\Lambda$ of $x$. Since
$x$ is stable we must have $|\Lambda|=1$. Let $\pi: \R^n \to (\spa
\Lambda)^{\perp}$ by the orthogonal projection and let 
$\Lambda' \df \pi(x)$. 

It suffices to show that both $\Lambda$
and $\Lambda'$ are stable. Indeed, if this holds then by the induction
hypothesis, both $\Lambda$
and $\Lambda'$ satisfy \equ{eq: 
  covrad} in their respective dimensions $k, n-k$, and by Lemma \ref{first
  Woods lemma}, so does $x$. To see that $\Lambda$ is stable, note
that any subgroup $\Lambda_0 \subset \Lambda$ is also a subgroup of
$x$, and since $x$ is stable, it satisfies $|\Lambda_0| \geq 1$.   To
see that $\Lambda'$ is stable, note that if $\Lambda_0 \subset
\Lambda'$ then $\widetilde{\Lambda_0} \df x \cap \pi^{-1}(\Lambda_0)$ is
a discrete subgroup of $x$ so satisfies $|\widetilde{\Lambda_0}| \geq
1$. Since $|\Lambda|=1$ and $\pi$ is orthogonal, we argue as in the
proof of \equ{eq: claim 1} to obtain
$$1 \leq |\widetilde{\Lambda_0}| = |\Lambda | \cdot |\Lambda_0| =
|\Lambda_0|,$$
so $\Lambda'$ is also stable, as required. 
\end{proof}

In \cite{mathieu}, it was shown that there is a unique local maximum
for covrad in dimension 1, none in dimensions 2--5, and a unique one in
dimension 6. Local maxima of covrad in dimension 7 are classified in
the manuscript \cite{mathieu2}; there are 2 such lattices. Thus
in total, in dimensions $n \leq 7$ there are 4 local maxima of the
function covrad. We
were informed by Mathieu Dutour-Sikiri\'c that these lattices all satisfy the covering radius bound
\equ{eq: covrad}. Thus Corollary  \ref{cor: Minkowski} yields another proof of
Minkowski's conjecture, in dimensions $n \leq 7$. In \cite[\S
7]{mathieu}, an infinite list of lattices (denote there by  $[L_n,
Q_n]$) is defined. The list consists of  one lattice in each
dimension $n \geq 6$, each of which is a local maximum for the
function covrad, and satisfies the bound \equ{eq:
  covrad}. It is expected that for each $n$, this lattice has the largest covering
radius among all
local maxima in dimension $n$. In light of Corollary \ref{cor: Minkowski}, the validity
of the latter assertion would imply  Minkowski's conjecture in all
dimensions.

\section{A volume computation}
\Name{sec: volume computation}
The goal of this section is the following. 
\begin{theorem}\Name{vol est theorem}
Let $m$ denote the $G$-invariant probability measure on
$\Xn$ derived from Haar measure on $G$, and let $\cS^{(n)} $
denote the subset of stable lattices in $\Xn$. Then $m\left(\cS^{(n)} \right)\lra 1$ as $n \to \infty$. 
\end{theorem}
Recalling the notation \equ{eq: k quantities}, for $k=1, \ldots, n-1$,
let 
$$
\cS^{(n)}_k(t) \defi\set{x\in \Xn: \al_k(x)\ge t}, \ \ \ \cS^{(n)}_k
\df \cS^{(n)}_k(1).
$$
It is clear that 
$\cS^{(n)}=\bigcap_{k=1}^{n-1}\cS^{(n)}_k$. 
In order to prove Theorem~\ref{vol est theorem} it is enough to prove
that 
\eq{1312}{
\max_{k=1, \ldots, n-1} m\left(\Xn \smallsetminus \cS^{(n)}_k \right)
= o\left(\frac1n \right),  
}
as this implies 
\begin{align*}
m\left(\cS^{(n)}\right )&= 1-m\left(\Xn\smallsetminus \cap_{k=1}^{n-1}
\cS^{(n)}_k \right)=1-m\left( \cup_{k=1}^{n-1} \left(\Xn\smallsetminus
  \cS^{(n)}_k \right) \right)\\ 
&\ge 1-\sum_{k=1}^{n-1}m \left(\Xn \smallsetminus
\cS^{(n)}_k\right)=1-(n-1)o\left(\frac1n\right)\overset{n\to\infty}{\lra}1. 
\end{align*}

We will actually prove a bound which is stronger than \equ{1312}, namely:
\begin{proposition}\Name{prop: strengthening volume} There is $C_1>0$
  such that if we set 
\eq{eq: choice of t}{
t_k=t(n,k) \df \left(\frac{n}{C_1} \right)^{\frac{k(n-k)}{2n} },
} 
then 
$$ \max_{k=1, \ldots, n-1} m\left(\Xn \smallsetminus
  \cS^{(n)}_k\left(t_k\right) \right) =o 
\left( \frac1n \right).
$$
In particular, $m\left(\bigcap_{k=1}^{n-1}
  \cS^{(n)}_k\left(t_k\right ) \right) \to_{n \to \infty} 1.$
\end{proposition}
Let
$$\gamma_{n,k}  \df \sup_{x \in \Xn} \alpha_k(x). 
$$
Recall that {\em Rankin's constants} or the {\em generalized Hermite's
  constants}, are defined as $\gamma_{n,k}^2$ (note that our notations
differ from traditional notations by a square root). 
Thunder \cite{Thunder} computed upper and lower bounds on
$\gamma_{n,k}$ and in particular established the growth
rate of $\gamma_{n,k}$. The numbers
$t(n,k)$ have the same growth rate. Thus
Proposition \ref{prop: strengthening volume} should
be interpreted as saying that the lattices in $\Xn$ for which the
value of each $\alpha_k$ is
close to the maximum possible value, occupy almost all of the measure of $\Xn$. 

The proof of Proposition \ref{prop: strengthening volume} relies on
Thunder's work, which in turn was based on a variant of Siegel's
formula~\cite{SiegelFormula} which relates the Lebesgue measure
on $\bR^n$ and the measure $m$ on $\Xn$. We now review Siegel's
method and Thunder's results.
 
In the sequel we consider $n \geq 2$ and $k \in \{1,
\ldots, n-1\}$ as fixed and omit, unless there is risk of confusion,
the symbols $n$ and $k$ from the notation.  
Consider the (set valued) map $\Phi=\Phi^{(n)}_k$ that assigns to
each lattice $x\in \Xn$ the following subset 
of $\wedge^k\bR^n$:
$$\Phi (x)\defi\set{\pm w_\Lam:\Lam \subset x\textrm{ a
    primitive subgroup with } r(\Lam)=k},$$ 
where $w_\Lam\defi v_1\wedge\dots\wedge v_k$ and $\set{v_i}_1^k$ 
forms a basis for $\Lam$ (note that $w_\Lam$ is well defined up to
sign, and $\Phi(x)$ contains both possible choices). 
Let $$\mathscr{V} = \mathscr{V}^{(n)}_k \defi
\set{v_1\wedge\dots\wedge v_k: v_i\in\bR^n} \sm \{0\}$$ 
be the variety of pure tensors in $\wedge^k\bR^n$. 
For any a compactly supported bounded Riemann integrable function $f$
on $\mathscr{V}$ set 
\eq{eq: finite sum}{\hat{f}: \Xn \to \R, \ \ \ \ 
  \hat{f}(x)\defi\sum_{w\in\Phi (x)}f(w).}
Then it is known  (see \cite{Weil}) that the (finite) sum \equ{eq: 
  finite sum} 
defines a function in $L^1(\Xn, m)$. 
Let  $\theta = \theta^{(n)}_k$ denote the Radon measure on $\mathscr{V}$
  defined by 
\begin{equation}\label{1420}
\int_{\mathscr{V}} f d\theta \defi\int_{\Xn} \hat{f} \, dm, \text{  for 
  \ } f\in C_c(\mathscr{V}).
\end{equation}

In this section we write $G=G_n \df \SL_n(\R)$. 
There is a natural transitive action of $G_n$ on 
$\mathscr{V} 
$ and the stabilizer of  
$e_1\wedge\dots\wedge e_k$ is the subgroup 
$$H= H^{(n)}_k \df \left\{ \smallmat{A&B\\0&D} \in G: 
A \in G_{k} , D \in G_{n-k} \right \}. $$
We therefore obtain an identification $\mathscr{V}\simeq G/H$ and view
$\theta$ as a measure on $G/H$.  

It is well-known 
(see e.g.~\cite{Raghunathans_book}) that up to a proportionality
constant there exists a unique $G$-invariant measure 
$m_{G/H}$ on $G/H$; moreover, given Haar
measures $m_{G}, m_{H}$ on $G$ and $H$ respectively, there is a
unique normalization of $m_{G/H}$ such that 
for any $f\in L^1(G,m_G)$
\eq{1440}{
\int_G f \, dm_G =\int_{G/H}\int_{H} f(gh) dm_{H}(h)dm_{G/H}(gH).
}
We choose the Haar measure $m_G$ so that it
descends to our probability measure $m$ on $\Xn$;  similarly, we  
choose the Haar measure $m_{H}$ so that the periodic orbit
$H\bZ^n \subset \Xn$ has volume 1. These choices of Haar measures
determine our measure $m_{G/H}$ unequivocally. 
It is clear from the defining formula~\eqref{1420} that $\theta$ is
$G$-invariant and therefore 
the two measures $m_{G/H}, \theta$ are proportional. In fact (see
\cite{SiegelFormula} for the case $k=1$ and  \cite{Weil} for the general case), 
\eq{eq: Siegel normalization}{m_{G/H} = \theta.
}

For $t>0$, 
let $\chi=\chi_t:\mathscr{V}\to\bR$ be the restriction to $\mathscr{V}$ of the
characteristic function of the ball of radius $t$ around the origin, in $\wedge^k\bR^n$. 
Note that  
$\hat{\chi}(x)=0$ if and only if $x\in \cS^{(n)}_k(t)$ and furthermore,
$\hat{\chi}(x)\ge 1$ if $x\in \Xn\smallsetminus \cS^{(n)}_k(t)$.  
It follows that
\eq{eq: using chi}{m\left(\Xn\smallsetminus
\cS_k^{(n)}(t)\right)\le\int_{\Xn}\widehat{(\chi_t)} dm =\int_{\mathscr{V}}\chi_t
d\theta.
}

Let $V_j$ denote the volume of the Euclidean unit ball in
$\bR^j$ and let $\zeta$ denote the Riemann zeta function.  We will use
an unconventional convention $\zeta(1)=1$, which will make our
formulae simpler. 
For $j \geq 1$, define 
$$
R(j) \df 
 \frac{j^2 V_j}{ \zeta(j)}  
$$
and 
$$B( n,k)\defi \frac{\prod_{j=1}^nR(j)}{\prod_{j=1}^k R(j)\prod_{j=1}^{n-k}R(j)}.$$
The following calculation was carried out in~\cite{Thunder}.
\begin{theorem}[Thunder]{\label{Thunder}}
For $t>0$, we have 
$$\int_{\mathscr{V} } \chi_t \, dm_{G/H}
=B( n,k)\frac{ t^n}{n}.$$ 
\end{theorem} 

We will need to bound $B( n,k)$. 
\begin{lemma}\Name{lem: bound on Bin}
There is $C> 0$ so that for all large enough $n$ and
all $k=1, \ldots, n-1$, 
\eq{eq: bound on Bin}{B( n,k)\leq
  \left(\frac{C}{n}\right)^{\frac{k(n-k)}{2}}.
}
\end{lemma}
\begin{proof}In this proof $c_0, c_1, \ldots$ are constants independent of $n, k, j$. 
Because of the symmetry $B(n,k)=B(n,n-k)$ it is
enough to prove \equ{eq: bound on Bin} with $k\leq \frac{n}{2}.$
Using 
the formula
$V_j=\frac{\pi^{j/2}}{\Ga\left(\frac{j}{2}+1\right)}$ we obtain 
%
\begin{align*}
B(n,k)&=\prod_{j=1}^k\frac{R(n-k+j)}{R(j)}
=\prod_{j=1}^k\frac{\zeta(j)(n-k+j)^2\frac{\pi^{(n-k+j)/2}}{\Ga(\frac{n-k+j}{2}+1)}}
{\zeta(n-k+j)j^2\frac{\pi^{j/2}}{\Ga(\frac{j}{2}+1)}}\\
&=\prod_{j=1}^k \frac{\zeta(j)}{\zeta(n-k+j)}\cdot\pa{\frac{n-k+j}{j}}^2\cdot\pi^{\frac{n-k}{2}}\cdot
\frac{\Ga(\frac{j}{2}+1)}{\Ga(\frac{n-k+j}{2}+1)}.
\end{align*}
Note that $\zeta(s) \geq 1$ is a decreasing function of $s>1$, so
(recalling our convention $\zeta(1)=1$) 
$\frac{\zeta(j)}{\zeta(n-k+j)} \leq c_0 \df \zeta(2)$.
It follows that for all large enough $n$ and 
for any $1\le j\le k, $ 
\eq{eq: estimate first part}{
\frac{\zeta(j)}{\zeta(n-k+j)}\cdot\pa{\frac{n-k+j}{j}}^2\cdot\pi^{\frac{n-k}{2}}\le c_0
n^2\pi^{\frac{n-k}{2}}\le 4^{\frac{n-k}{2}}.
}
According to Stirling's formula, there are positive constants $c_1,
c_2$ such that for all $x \geq 2$, 
$$
c_1 \sqrt{\frac{2\pi}{x}}\left(\frac{x}{e} \right)^x \leq \Gamma(x)
\leq c_2 \sqrt{\frac{2\pi}{x}}\left(\frac{x}{e} \right)^x.
$$
We set $u \df \frac{j}{2}+1$ and $v \df \frac{n-k}{2} $, so that 
$u+v \geq \frac{n-1}{4}$, 
and obtain  
\eq{eq: estimate second part}{
\begin{split}
\frac{\Ga(\frac{j}{2}+1)}{\Ga(\frac{n-k+j}{2}+1)} & =
\frac{\Ga(u)}{\Ga(u+v)} \leq \frac{c_2}{c_1}
\sqrt{\frac{u+v}{u}}\frac{u^u}{(u+v)^{u+v}} \frac{e^{u+v}}{e^u} \\
& \leq c_3 e^v \frac{u^{u-1/2}}{(u+v)^{u+v-1/2}} = c_3
\left(\frac{e}{u+v}\right)^v \frac{1}{\left(1+\frac{v}{u}
  \right)^{u-1/2}}, 
\\
& \leq c_3 \left( \frac{4e}{n-1} \right)^{\frac{n-k}{2}}. 
\end{split}
}
Using \equ{eq: estimate first part} and \equ{eq: estimate
  second part} we obtain
$$
B( n,k) \leq \left[c_3 4^{\frac{n-k}{2}}
  \left(\frac{4e}{n-1}\right)^{\frac{n-k}{2}} \right]^k = \left[ c_3
  \left(\frac{16e}{n-1} \right)^{\frac{n-k}{2}} \right]^k.
$$
So taking $C > 16c_3 e$ 
we obtain \equ{eq: bound on
  Bin} for all large enough $n$. 
\end{proof}
\begin{proof}[Proof of Proposition \ref{prop: strengthening volume}]
Let $C$ be as in Lemma \ref{lem: bound on Bin} and let $C_1>C$. 
Then by \equ{eq: using chi}, \equ{eq: Siegel normalization} and 
Theorem~\ref{Thunder}, for all sufficiently large $n$ we have
\[\begin{split}
m\left(\Xn \smallsetminus \cS^{(n)}_k(t_k) \right) & \leq
B(n,k) \frac{t_k^n}{n} \\
& \leq \frac1n \left(\frac{C}{n} \right)^{\frac{k(n-k)}{2}}
\left(\frac{n}{C_1} \right)^{\frac{k(n-k)}{2}} =  \frac1n \left(\frac{C}{C_1} \right)^{\frac{k(n-k)}{2}}.
\end{split}
\]
Multiplying by $n$ and taking the maximum over $k$ we obtain 
$$
n \, \max_{k=1, \ldots, n}  m\left(\Xn \smallsetminus
  \cS^{(n)}_k(t_k) \right) \leq \left(\frac{C}{C_1}
\right)^{\frac{n-1}{2}} \to_{n\to \infty} 0.
$$
\end{proof}

\section{Closed $A$-orbits and well-rounded lattices} \Name{sec:
  closed orbits} 
It is an immediate consequence of Theorem \ref{thm: main} that any
closed $A$-orbit contains a stable lattice. The purpose of this
section is to show that the same is true for the set of well-rounded
lattices. 
Note that this was proved by McMullen for {\em compact} orbits but
for general  closed orbits, does not follow from his results. Our
proof relies on previous work of Tomanov and the second-named author
\cite{TW}, on \cite{gruber}, and on a covering result (communicated to
the authors by Michael Levin), whose proof is given in the appendix to
this paper. 

\begin{theorem}\Name{thm: closed orbits}
For any $n$, any closed orbit $Ax \subset \Xn$ contains a well-rounded lattice. 
\end{theorem}

We will require the following topological result which generalizes 
Theorem \ref{topological input}.  
Let $s,t$ be
natural numbers, and let $\Delta$ denote the 
$s$-dimensional simplex, which we think of concretely as $\mathrm{conv} (\E_1,
\ldots, \E_{s+1})$. 
We will discuss covers of $M \df \Delta \times \R^t $, and give conditions
guaranteeing that such a cover must cover a point at least $s+t+1$
times. 
For
$j=1, \ldots, s+1$ let $F_j$ be the face 
of $\Delta$ opposite to $\E_j$, that is $F_j = \mathrm{conv} (\E_i: i
\neq j)$. Also let $M_j \df F_j \times  \R^t $ be the corresponding
subset of $M$. 

\begin{theorem}
\Name{thm: covering}
Suppose that  $\cU$ is a cover of $M$ 
by open sets 
satisfying the following conditions:
\begin{enumerate}[(i)]
\item\Name{09301}
For any connected component $U$ of any element of  $\cU$ there exists $j$ 
such that $U \cap M_j = 
\varnothing.$  
\item\Name{09302}
There is $R$ so that for 
any connected component $U$ of the intersection of $k \leq s+t$ distinct 
elements of $\cU$, 
the projection of $U$ to $\R^t$ is $(R, s+t-k)$-almost
affine.
\end{enumerate}
Then
there is a point of $M $ which is covered at least
$s+t+1$ times.

\end{theorem}

Note that hypothesis~\eqref{09302} is trivially satisfied when $k \leq s $,
since any subset of $\R^t$ is $(1, t)$-almost affine. 
Note also that Theorem \ref{topological input} is the case $s=0$ of this
statement. We give the proof of Theorem \ref{thm: covering} in the appendix. 

We will need some preparations in order to deduce Theorem~\ref{thm:
  closed orbits} from Theorem~\ref{thm: covering}. For $1\le d\le n$,
let $$\tb{I}^n_d\defi\set{1\le i_1<\dots<i_d\le n}$$  
denote the collection of multi-indices of length $d$ and 
for $J = (i_1, \ldots, i_d)\in\tb{I}^n_d$ let 
$e_J \df e_{i_1} \wedge \cdots \wedge e_{i_d}.
$
We equip $\bigwedge_1^d\bR^n$ with the inner product with
respect to which $\{e_J\}$ is an orthonormal basis, and denote by $\cE_{d,n}$
the quotient of $\bigwedge_1^d\bR^n$ by the equivalence relation $w 
\sim -w$. Note that the product of an element of $\cE_{d,n}$ with a
positive scalar is well-defined. We will (somewhat imprecisely) refer to elements of $\cE_{d,n}$
as vectors. Given a subspace $L \subset \bR^n$  with $\dim L= d$,
we denote by 
  $w_L\in \cE_{d,n}$ the image of a vector of norm one in
 $ \bigwedge_1^d L.$
If $\Lam \subset \bR^n$ is a discrete subgroup of rank $d$, we
  denote by $w_\Lam\in \cE_{d,n}$ 
the image of the vector 
$v_1\wedge\dots\wedge v_d,$ where $\set{v_i}_1^d$ forms a basis for
$\Lam$. The reader may verify that these vectors are well-defined and
satisfy $w_{\Lam} = |\Lam| w_L$ where $L = \spa \Lam$.  
We denote the natural action of $ G$ on $\cE_{d,n}$ arising from the
$d$-th exterior power of the linear action on $\bR^n$, by $(g, w)
\mapsto gw$. 
Given a subspace $L \subset \bR^n$ and a discrete subgroup $\Lam$ we set 
$$A_L\defi\set{a\in A:
    aw_L=w_L} \text{ and } A_\Lam \df \{a \in A: aw_\Lam = w_\Lam\}.$$
Note that the requirement  
$aw_L=w_L$ is equivalent to saying that $aL=L$ and $\det(a|_L)=1$.
Given a flag 
\begin{equation}\label{flag}
\crly{F}=\set{ 0 \varsubsetneq L_1\varsubsetneq\dots\varsubsetneq L_k\varsubsetneq \bR^n}
\end{equation} 
(not necessarily full), let 
$A_{\crly{F}}\defi \bigcap_i A_{L_i}.$
The {\em support} of an element $w \in \cE_{d,n}$ is the subset of
$\tb{I}^n_d$ for which the corresponding coefficients of an element of
$\bigwedge^d\bR^n$ representing $w$ are nonzero, and we write 
$\on{supp}(L)$ or $\on{supp}(\Lam)$ for the supports of $w_L$ and
$w_{\Lam}$. For $J = \set{i_1<\dots<i_d } \in \tb{I}^n_d$, set $\bR^J
\df \spa (e_{i_j})$ and define the multiplicative characters 
$$ \chi_J: A \to \R^*, \ \chi_J(a) \defi \det
(a|_{\bR^J}).$$
Then 
for any subspace $L\subset \bR^n$, 
\eq{1636}{A_L=\bigcap_{J \in \supp (L)} \ker \chi_J}
(and similarly for discrete subgroups $\Lam$). 
As in \S \ref{establishing
  topological input} we fix an invariant metric on $A$. In order to
verify hypothesis \eqref{09302} of Theorem \ref{thm: covering}, we will need the following
lemmas (cf. \cite[Theorem 6.1]{McMullenMinkowski}):
\begin{lemma}\label{bdd dist from stab}
Let $T \subset A$ be a closed subgroup and let $x\in\cL_n$ be a
lattice with a
compact $T$-orbit. Then for any $C>0$ there exists  
$R>0$ 
such that for any collection $\set{\Lam_i}$ of subgroups of $x$, there exists $b\in A$ such that 
\begin{equation}\label{2052}
\left \{a\in T:\forall i\; \norm{aw_{\Lam_i}}\le C \right\} \subset R
\text{-neighborhood of } b\pa{\bigcap_iA_{\Lam_i}}.
\end{equation}
\end{lemma}
\begin{proof}
We will identify $A$ with its Lie algebra $\mathfrak{a}$ via the
exponential map, and think of the subgroups $A_\Lam$ as
subspaces. 
By \equ{1636} only finitely many subspaces arise as $A_\Lam$.
In particular, given a collection of discrete subgroups
$\set{\Lam_i}$, the angles between the spaces they span (if nonzero) are bounded
below. Therefore
there exists a function $\psi:\bR\to\bR$ with $\psi(R)\to_{R\to\infty}\infty$, such that   
\begin{align}\label{336}
&\set{a\in A: \forall\; J\in\cup_i \on{supp}(w_{\Lam_i}),\; \psi(R)^{-1}\le \chi_J(a)\le \psi(R)}\subset\\
\nonumber &\set{a\in A: d(a,\cap_iA_{\Lam_i})\le R}.
\end{align}

Since $Tx$ is compact,  there exists a compact subset
$\Om\subset T$ 
such that for any $a\in T$ there exists $b = b(a) \in T$ satisfying $bx=x$ and $b^{-1}a\in\Om$.
It follows that there exists $M \geq 1$ such that:
\begin{enumerate}[(I)]
\item\label{1708} for any subspace $L$, $||bw_L||\le M||aw_L||$. 
\item\label{1747} for any multi-index $J$, $\chi_J(ba^{-1})\le M$.
\end{enumerate}

Given $C>0$, let $C' \df MC$ and consider the finite set 
$$\crly{S}\defi \set{\Lam \subset x: ||w_{\Lam}||\le C'}.$$
For any $\Lam\in\crly{S}$ write 
$w_{\Lam}=\sum_{J\in\on{supp}(w_{\Lam})} \al_J(\Lam) e_J.$ 
Let $\vre>0$ be small enough so that
\begin{align*}
\vre&<\min\set{\av{\al_J(\Lam)}:\Lam\in\crly{S}, J\in\on{supp}(w_{\Lam})},
\end{align*}
and choose $R$ large enough so that $\psi(R)>C'/\vre$. We claim that 
for any $\set{\Lam_i}\subset \crly{S}$,
\begin{equation}\label{2053}
\set{a\in T:\forall i\; \norm{aw_{\Lam_i}}\le C}\subset \set{a\in T: d(a, \cap_i A_{\Lam_i})\le R}.
\end{equation}
To prove this claim, suppose $a$ is an element on the left hand side
of~\eqref{2053}. 
By~\eqref{336}
it is enough show that for any $J\in\cup_i\on{supp}(\Lam_i)$ we have
$\psi(R)^{-1}\le \chi_J(a)\le\psi(R)$.  
Since the coefficient of $e_J$ in the
expansion of $aw_{\Lam_i}$ is $\chi_J(a)\al_J(\Lam_i)$ and since 
$||aw_{\Lam_i}||\le C$, we have 
$$\chi_J(a)\le
\frac{C}{|\al_J(\Lam_i)|}\le\frac{C}{\vre}\le \psi(R).$$ 
On the other hand, letting $b = b(a)$ 
we have $b\Lam_i\in\crly{S}$ from \eqref{1708}, and 
\begin{align*}
 \vre\le |\al_J(b\Lam_i)| & =\chi_J(b) |\al_J(\Lam_i)| \ \Longrightarrow \ \chi_J(b^{-1})\le C/\vre \\
&\overset{\textrm{\eqref{1747}}}{\Longrightarrow} \ \chi_J(a^{-1})=\chi_J(a^{-1}b)\chi_J(b^{-1})\le C'/\vre\le\psi(R),
\end{align*}
which completes the proof of \eqref{2053}.

Let $\set{\Lam_i}$ be any collection of subgroups of $x$
and assume that the set on the left hand side of ~\eqref{2052} is non-empty. That is, there 
exists $a_0\in T$ such that for all $i$, $||a_0w_{\Lam_i}||\le C$. 
Let $b = b(a_0)\in T$, and set $\Lam'_i \df b\Lam_i$. It follows that $\set{\Lam'_i}\subset\crly{S}$ 
and so 
\begin{align*}
\set{a\in T:\forall i\norm{aw_{\Lam_i}}\le C} 
&=b\set{a\in T: \forall i \norm{aw_{\Lam'_i}}\le C}\\
&\stackrel{\eqref{2053}}{\subset} b \set{a\in T: d(a,\cap_i A_{\Lam_i'})\le R}\\
&= \set{a\in T: d(a, b\pa{\cap_i A_{\Lam_i}})\le R},
\end{align*}
where in the last equality we used the fact that
$A_{\Lam_i'}=A_{\Lam_i}$ because $A$ is commutative.
\end{proof}

\begin{lemma}\Name{lem: flag}
Let $\crly{F}$ be a flag as in \eqref{flag} and let
$A_{\crly{F}}$ be its stabilizer. Then $A_{\crly{F}}$ is of
co-dimension  
$\ge k$ in $A$.
\end{lemma}
\begin{proof}
Given a nested sequence of multi-indices
$J_1\varsubsetneq\dots\varsubsetneq J_k$ it is clear that the subgroup
$$
\bigcap_{i=1}^k \ker \chi_{J_i}
$$
is of co-dimension $k$ in $A$. In light of \eqref{1636},
it suffices to prove the following claim: 
\quad\\
\noindent \textit{Let $\crly{F}$ be a flag as
  in~\eqref{flag} with $d_i\defi \dim L_i$. Then there is a nested
  sequence 
of multi-indices $J_i\in\tb{I}^n_{d_i}$ such that $J_i\in\on{supp}(L_i)$.}
\quad\\


In proving the claim we will assume with no loss of generality that
the flag is complete. Let $v_1,\dots, v_n$ be a basis of $\bR^n$ such
that  $L_i=\on{span}\set{v_j}_{j=1}^i$ for 
$i=1,\dots, n-1.$  
Let $T$ be the $n\times n$ matrix whose columns are $v_1, \dots, v_n$.
Given a multi-index $J$ of length $\av{J}$, we denote by $T_J$
the square matrix of dimension $\av{J}$ obtained  
from $T$ by deleting the last $n-\av{J}$ columns and the rows
corresponding to the indices not in $J$.  Note that with this
notation, possibly after replacing some of the $v_i$'s by their scalar
multiples, each $w_{L_d}$ is the image in
$\cE_{d,n}$ of 
\begin{equation}\label{1159}
v_1\wedge\dots\wedge v_d=\sum_{J\in \tb{I}^n_d}  (\det T_J) e_J.
\end{equation}
In particular,  $J\in\on{supp}(L_d)$ if and only if $\det T_J\neq 0$.

Proceeding inductively in reverse, we construct the nested sequence
$J_d$ by induction on $d =n, \ldots, 1$. Let $J_n=\set{1,\dots, n}$ so that
$T=T_{J_n}$.  
Suppose we are given  multi-indices
$J_{n}\supset\dots\supset J_{d+1}$ such that  
$J_i\in\on{supp}(w_{L_i})$ 
for $i=n,\dots, d+1$. We want to define now a multi index
$J_d\in\on{supp}(w_{L_d})$ which is  
contained in $J_{d+1}$. By~\eqref{1159}, $\det T_{J_{d+1}}\neq 0$. When
computing $\det T_{J_{d+1}}$ by expanding the last column we express
$\det T_{j_{d+1}}$ 
as a linear combination of $\set{\det T_J:J\subset J_{d+1},
  \av{J}=d}$. We
conclude that there must exist at least one multi-index $J_d\subset
J_{d+1}$  for which $\det T_{J_d}\ne 0$. In turn, by~\eqref{1159} 
this means that $J_d\in\on{supp}(w_{L_d})$. This finishes the proof
of the claim. 
\end{proof}

The following notation is analogous to Definition \ref{bn}.
Given a lattice $x\in \Xn$ and $\del>0$ let 
\begin{align*}
\on{Min}^*_{\del}(x)&\defi\set{v \in x \sm \{0\} : \|v\|<(1+\del)\al_1(x)}.\\
\tb{V}^*_{\del}(x)&\defi\on{span}\on{Min}^*_{\del}(x).\\
\dim^*_\del(x)&\defi\dim\tb{V}^*_{\del}(x).
\end{align*}
Finally, for $\vre>0$, let $\cU^{(\vre)}=\set{U_j^{(\vre)}}_{j=1}^n$
be the collection of open subsets of $A$ defined by  
\eq{eq: cover}{ U_j = U_j^{(\vre)} \df \{a \in A: \text{for all }
  \delta \text{ in a neighborhood of }
j\vre, \, \dim^*_{\del} (ax) =j \}. }

Similarly to the discussion in Lemma \ref{lem: positive inradius} we see that 
$\cU^{(\vre)}$ is an open cover of $A$.
\begin{proof}[Proof of Theorem \ref{thm: closed orbits}]  
The strategy of proof is very similar to that of Theorem~\ref{thm: main}.  We consider
covers $\cU^{(\varepsilon)}$ of $A$  and use Theorem~\ref{thm: covering}
to deduce that $U_n^{(\varepsilon)}$ is non-empty. 
The first step towards applying Theorem~\ref{thm: covering} is to find a decomposition
 $A\simeq\bR^{n-1}=\bR^s\times\bR^t$ 
and a simplex $\Del\subset \bR^s$, so that the restriction of the cover to
$\Del\times\bR^t$ satisfies the two hypotheses of Theorem~\ref{thm: covering}.

According to \cite{TW, gruber}, there is a
decomposition $A = T_1
\times T_2$ and a direct sum decomposition $\R^n = \bigoplus_1^{d} V_i$
such that the following hold:
\begin{itemize}
\item
Each $V_i$ is spanned by some of the standard basis vectors. 
\item
$T_1$ is the group of linear transformations 
which act on each $V_i$ by a homothety, preserving Lebesgue measure on
$\R^n$. In particular $s \df \dim T_1 = d-1$. 
\item
$T_2$  is the group of diagonal (with respect to the standard basis)
matrices whose restriction to each $V_i$ has determinant 1.  
\item $T_2 x$ is compact and $T_1 x$ is divergent;
  i.e. $Ax \cong T_1 \times T_2/(T_2)_x$, where
  $(T_2)_x \df \{a \in T_2: ax= x\}$. 
\item 
Setting $\Lambda_i \df V_i \cap x$, each $\Lambda_i$ is a
lattice in $V_i$, so that $\bigoplus \Lambda_i$ is of finite index in
$x$. 
\end{itemize}

For $a \in T_1$ we write $\chi_i(a)$ for the number
satisfying $av = e^{\chi_i(a)}v$ for all $v \in V_i$. Thus each $\chi_i$
is a homomorphism from $T_1$ to the additive group of real
numbers. The mapping $a \mapsto \bigoplus_i \chi_i(a)
\mathrm{Id}_{V_i}$, where $\mathrm{Id}_{V_i}$ is the identity map on
$V_i$,  is nothing but the logarithmic map of $T_1$ and it endows
$T_1$ with the structure of a vector space. In particular we can discuss the
convex hull of subsets of $T_1$. 
For each $\rho$ we let
$$\Delta_\rho \df \{a \in T_1: \max_i \chi_i(a) \leq \rho\}.$$  
Then
$\Delta_\rho = \conv (b_1, \ldots, b_d)$ where $b_i$ is the diagonal
matrix acting on each $V_j, j \neq i$ by multiplication by $e^\rho$, and
contracting $V_i$ by the appropriate constant ensuring that $\det b_i
=1$. 

Let $P_i : \R^n \to 
V_i$ be the natural projection associated with the decomposition $\R^n =
\bigoplus V_i$. Each $P_i(x)$ is of finite index in $\Lambda_i$ and
hence discrete in $V_i$. Moreover, the orbit $T_2 x$ is compact, so
for each $a \in T_2$ there is $a'$ belonging to a bounded subset of $T_2$
such that $ax=a'x$. This implies that there is
$\eta>0 $ such that for any $i$ and any $a \in T_2$, if $v \in ax$ and $P_i(v) \neq 0$
then $\| P_i(v)\| \geq \eta$. Let $C>0$ be large enough so that
$\alpha_1(x') \leq C$ for any $x' \in \Xn$.  Let
$\rho$ be large enough so that 
\eq{eq: choice of R}{e^\rho\eta > 
2C. }
We restrict the covers $\mathcal{U}^{(\vre)}$
(where 
$\vre \in (0,1/n)$) to $\Delta_\rho \times
T_2$ and apply Theorem \ref{thm: covering} with $t \df \dim T_2 = n-d$. 
 If we show that the 
  hypotheses of Theorem \ref{thm: covering} are 
  satisfied for each
  cover $\mathcal{U}^{(\vre)},$ we will obtain $U_n^{(\vre)} \neq \varnothing.$ Then, taking $\vre_j
  \to 0$ and applying a
  compactness argument, we find a well-rounded lattice in 
  $(\Delta_\rho \times T_2)x$. 

 Let   $U$ be a connected  subset of 
  $U_k^{(\vre)} \in \cU^{(\vre)}$. Repeating the arguments proving Lemma
  \ref{flat things}, or
  appealing to \cite[\S7]{McMullenMinkowski}, we see that the
  $k$-dimensional subspace 
  $L\defi a^{-1}\tb{V}^*_{k\vre} (ax)$ 
  as well as the discrete subgroup $\Lam\defi L\cap x$
  are independent of the choice of $a \in U$. 
  By definition of $U_k^{(\vre)}$, for any $a \in U$, $a \Lam$ contains $k$ vectors $v_i = v_i(a),
  i=1, \ldots,
  k$ which span $aL$
  and satisfy \eq{eq: vi satisfy}{
\|v_i\| \in [r, (1+k\vre)r ], \ \ \text{where \ } r \df  \alpha_1(ax).
 } 
 
In order to 
verify hypothesis~\eqref{09301} of Theorem \ref{thm:
    covering}, we need to show that there 
is at least one $j$ for which $U \cap M_j  = 
\varnothing$. Let $P_1, \ldots, P_d$ be the projections above. Since 
 $\ker P_1 \cap \cdots \cap \ker P_d = \{0\}$ and $\dim L = k \geq 1$, it suffices to show that
 whenever $U \cap M_j \neq \varnothing$, $L \subset \ker P_j$. 
The face $F_j$ of $\Delta_\rho$ consists of those elements $a_1 \in T_1$
which expand vectors in $V_j$ by a factor of 
$e^\rho$. If $U \cap M_j \neq \varnothing$ then there is $a \in T_2, a_1
\in F_j$ so that $a_1a \in U$. Now \equ{eq: choice of R}, \equ{eq: vi
  satisfy} and the choice of $\eta$ and $C$  ensure that 
the vectors $v_i = v_i(a_1a)$ satisfy $P_j(v_i)=0$. Therefore $L \subset
\ker P_j$.  

It remains to 
verify hypothesis~\eqref{09302} of Theorem \ref{thm: covering}. 
Let $U$ be a connected subset of an intersection $U_{i_1}\cap\dots\cap
U_{i_k}\cap(\Del_\rho\times T_2)$ and let  
$L_{i_j}\defi a^{-1}\tb{V}^*_{i_j\vre} (ax)$ and $\Lam_{i_j}\defi L_{i_j}\cap x$. 
As remarked above, $L_{i_j},\Lam_{i_j}$ are independent of $a\in U$.

By the definition of the $L_{i_j}$'s we have that $L_{i_j}\varsubsetneq L_{i_{j+1}}$ and so they form 
a flag $\crly{F}$ as in~\eqref{flag}. Lemma~\ref{lem: flag} applies and we deduce that 
\begin{equation}\label{1640}
A_{\crly{F}}=\cap_{j=1}^k A_{L_{i_j}} \textrm{ is of co-dimension}\ge k\textrm{ in }A.
\end{equation} 
For each $a\in U$ and each $j$ let $\set{v^{(j)}_\ell(a)}\in a\Lam_{i_j}$ be the vectors spanning 
$aL_{i_j}$ which satisfy~\eqref{eq: vi satisfy}. Let 
$u^{(j)}_\ell(a)\defi a^{-1} v^{(j)}_\ell\in\Lam_{i_j}$. Observe that:
\begin{enumerate}[(a)]
\item\label{02281} $\on{span}_{\bZ}\set{u^{(j)}_\ell(a)}$ is of finite
  index in $\Lam_{i_j}$ and in particular, 
$u^{(j)}_1(a)\wedge\dots\wedge u^{(j)}_{i_j}(a)$ is an integer
multiple of $\pm w_{\Lam_{i_j}}$. As a consequence $||aw_{\Lam_{i_j}}||\le
||v^{(j)}_1(a)\wedge\dots\wedge v^{(j)}_{i_j}(a)||$. 
\item\label{02282} Because of~\eqref{eq: vi satisfy} we have that $ ||v^{(j)}_1(a)\wedge\dots\wedge v^{(j)}_{i_j}(a)||< C$ for some constant depending on $n$ alone.
\end{enumerate}
It follows from~\eqref{02281},\eqref{02282} and Lemma~\ref{bdd dist from stab}
that there exist $R>0$ and an element $b\in T_2$ so that
$$U\subset \Del_\rho \times \set{a\in T_2:\forall i_j,
  ||aw_{\Lam_{i_j}}||<C}\subset T_1\times \set{a\in T_2:
  d(a,bA_{\crly{F}})\le R}.$$ 
By~\eqref{1640} we deduce that 
if $p_2 : A \to T_2$ is the projection 
associated with the
decomposition $A = T_1 \times T_2$ then $p_2(U)$ is $(R',s+t-k)$-almost afine, where $R'$ depends only on $R,\rho$. This concludes the proof. 
\end{proof}

\appendix
\section{Proof of Theorem \ref{thm: covering}}
\Name{appendix: Levin}
Below $X$ will denote a second countable metric space.
We will use calligraphic letters like $\mathcal{U}$ for collections of
sets. The symbol 
 $\mesh ( \mathcal {A})$ will 
 denote the supremum of the diameters of the sets in
 $\mathcal A$.
The symbol $\Lb (\mathcal{A})$ will denote
 the Lebesgue number of a cover $\mathcal A$, i.e. the supremum of all
 numbers $r$ such that each ball of radius $r$ in $X$ is contained in
 some element of $\cA$. The symbol
 $\order(\mathcal{A})$ will denote the  
 largest number of distinct elements of $\mathcal A$ with non-empty
 intersection. 
\begin{definition}\Name{def: asdim}
 A collection $\{X_j\}_{j \in \crly{J}}$ of subsets of $X$ is said to be {\em uniformly of 
 asymptotic dimension $\leq n$}  if for every $r>0$ 
 there is $R >0$ such that for every $j \in \crly{J}$ there is 
  an open cover ${\mathcal X}_j $  of $X_j$ such that
\begin{itemize}
\item  $\mesh ( {\mathcal X}_j) \leq R$.
\item $\Lb ({\mathcal X }_j)> r$. 
\item$\order( {\mathcal X}_j)\leq n+1$. 
\end{itemize}

As an abbreviation we will sometimes write `asdim' in place of `asymptotic
dimension'. 
\end{definition}

Recall that a cover of $X$ is {\em locally finite} if every $x \in X$
has a neighborhood which intersects finitely many sets in the cover. 
We call the intersection of $k$ distinct elements of $\cA$ a {\em
  $k$-intersection,} and denote the union of all $k$-intersections by 
 $[{\mathcal A}]^{k}$. We will need the following two Propositions for
 the proof of Theorem~\ref{thm: covering}. We first prove  
 Theorem~\ref{thm: covering} assuming them and then turn to their proof.

\begin{proposition}
\Name{p2}
Let $\mathcal A$ be a locally finite open cover of a space $X$
such that $\order ({\mathcal A}) \leq m$ and the collection of components of
the  $k$-intersections 
of 
$\mathcal A$, $1 \leq k \leq m$,  is uniformly  of $\asdim \leq m-k$.
Then $\mathcal A$ can be refined by a uniformly bounded 
open cover of order at most  $m$.
\end{proposition}

 \begin{proposition}
 \Name{p3}
 Let $\Delta_1$ and $\Delta_2$ be simplices, $X=\Delta_1 \times \Delta_2$,
 $p_i : X \to \Delta_i$ the projections and 
  $\mathcal A$  a  finite open cover of $X$ such that
  for every $A \in \mathcal A$ and $i=1,2$ the set $p_i(A)$ does 
  not meet at least one of the faces of $\Delta_i$.
  Then $\order ({\mathcal A}) \geq \dim \Delta_1 + \dim \Delta_2 +1 $.
  \end{proposition}
 

 \begin{proof}[Proof of Theorem \ref{thm: covering}.] 
Let $m \df \dim M = s+t$, and suppose by contradiction that 
 $\order ({\mathcal U}) \leq  m$. Since every cover of $M$ has a
 locally finite refinement, there is no loss of generality in assuming
 that $\cU$ is locally finite. Replacing $\cU$
   with the set of connected components of elements of $\cU$, we may 
   assume that all elements of $\cU$ are connected. For any $r_0$, and
   any bounded set $Y$, the
   product space  $Y \times \R^d$ can be covered by a cover of order $d+1$ and
   Lebesgue number greater than $r_0$. Hence our hypothesis (ii) implies that for each
 $k =1, \ldots, m$, the collection of connected components of intersections of $k$
 distinct elements of $\cU$ is
 uniformly of
 asymptotic dimension at most $m-k$. 
Therefore we can apply Proposition
 \ref{p2} 
to assume that  $\mathcal U$ 
  is  uniformly bounded and of order at most $m$. 
  Take a sufficiently large $t$-dimensional simplex  $\Delta_1 \subset \R^t$ so that 
   the projection of every set  in $\mathcal U$  does not
  intersect at least one of the faces of $\Delta_1$. 
We obtain a contradiction to 
   Proposition \ref{p3}.
\end{proof}
For the proofs of Propositions \ref{p2}, \ref{p3} we will need some auxiliary
lemmas. 
\begin{lemma}\Name{lem: for p2}
Let $\{G_i: i \in \cI\}$ be a locally finite collection of open subsets
of a metric space $X$,
and let $Z$ be an open subset such that for each $i \neq j$, $G_i
\cap G_j \subset Z$. Then there are disjoint open subsets
$\tb{E}_i, i \in \crly{I},$ such that for any $i$
\eq{eq: inclusion}{G_i \sm Z \subset \tb{E}_i
\subset G_i.}   
\end{lemma}

\begin{proof}
For each $G_i$ and 
$x \in G_i \cap \partial
\, (G_i \sm Z)$ 
set   
$$
r_x \df \frac{1}{3} \inf_{j \neq i} d\left(x, \partial
\, \left(G_j  \sm Z\right) \right),
$$
where $d$ is the metric on $X$. The infimum in this
definition is in fact a minimum since $\{G_i\}$ is locally
finite. To see that it
is positive, suppose if possible that  $y_\ell \to x$
for a
sequence $(y_\ell) \subset \partial ( G_j \sm Z)$. 
Then there are $\tilde{y}_\ell \in G_j \sm Z$ with $d(y_\ell, \tilde{y}_\ell) \to 0$ so that
$\tilde{y}_\ell \to x$.  Since $G_i$ is open, for large enough $\ell$ we have
$\tilde{y}_\ell \in G_i$, contradicting the assumption that $G_i \cap G_j
\subset Z$. 
Now we set 
$$
\tb{E}_i \df \tb{E}' \cup \tb{E}'', \ \ \text{where \ \ }
\tb{E}' \df G_i \sm Z \ \ \text{and \ } \tb{E}'' \df G_i \cap
\bigcup_{x \in  G_i \cap \partial
\, (G_i \sm Z)} B(x,r_x).
$$
Clearly each $\tb{E}_i$ satisfies \equ{eq: inclusion}, and it is
open since $\tb{E}''$ is open and covers the boundary points of
$\tb{E}'$. To show that the sets $\tb{E}_i$ are 
disjoint, suppose if possible that $z \in \tb{E}_i \cap
\tb{E}_j$. Then there are $x \in G_i, y \in G_j$ such that $z \in
B(x, r_x) \cap B(y, r_y)$. Supposing with no loss of generality that
$r_x \geq r_y$ we find that  
$$d(x,y) \leq d(x,z)+d(z,y) \leq 2r_x \leq \frac23 d \left(
  x, \partial \left(G_j \sm Z \right)\right),
$$
which is impossible. 
\end{proof}
We denote the nerve of a cover $\cA$ by 
  $\nerve (\mathcal{ A}),$ and consider it with the metric topology
  induced by barycentric coordinates.
  Given a partitition of unity subordinate 
  to a cover $\mathcal{A}$ of $X$, there is a standard construction of
  a map $X \to \nerve(\cA)$; such a
  map is called a {\em canonical map}. 

 \begin{lemma}
\Name{p1}
Let a space $Y$ be the union of two open subsets $\mathbf{D}$ and $\mathbf{E}$, and
let $\mathcal D$ and $\mathcal E$ be open covers of $\mathbf{D}$ and
$\mathbf{E}$ respectively, 
with bounded $\mesh$ and $\order$, and such that if $C\subset \mathbf{D}\cap\mathbf{E}$ is a connected subset 
contained in an element of $\cD$, then it is contained in an element of $\cE$.
Then, there is an open cover $\cY$ of $Y$ such that:
\begin{enumerate}
\item The cover $\mathcal{Y}$ refines
$\cD\cup \cE$.
\item $\mesh ({\mathcal Y}) \leq 
\max \left( \mesh(\cD), \mesh(\cE) \right ).$ 
\item $\order ({\mathcal Y}) \leq \max \left ( \order ({\mathcal D})+1, \order
\left({\mathcal E}\right) \right )$.
\end{enumerate}
\end{lemma}

\begin{proof} 
Let 
$\order ({\mathcal D})=n+1$, let $X \df \nerve ({\mathcal D})$, and  let $\pi : \mathbf{D} \to X$
be a canonical map. Take an open cover of ${\mathcal X}$ of $X$ such that
$\order ({\mathcal X})\leq n+1$ and $\pi^{-1}({\mathcal X})$ refines ${\mathcal D}$.
Let $f : Y \to [0,1]$ be a continuous map such that
$f|_{Y \smallsetminus \mathbf{E}} \equiv 0$ and $f|_{Y\smallsetminus
  \mathbf{D}} \equiv 1$. Set
 $Y' \df f^{-1} \left( \left[\frac{1}{3},\frac{2}{3} \right ]\right )$
and 
$$g : Y' \to Z \df X\times \left [\frac{1}{2},\frac{2}{3} \right], \ \ \
g(y) \df (\pi(y), f(y)).$$
Since $\dim Z \leq n+1$ there is an open cover $\mathcal Z$ of $Z$
such that 
$\order ({\mathcal Z}) \leq n+2$,
the projection of  $\mathcal Z$ to $X$ refines ${\mathcal X}$
and the projection of $\mathcal Z$ to 
$\left[\frac{1}{2},\frac{2}{3} \right]$ is of $\mesh < 1/3$.
Let ${\mathcal Y}'$ denote the collection of connected components of
sets $\{g^{-1}(W): W \in \mathcal{Z}\}$. By construction $\cY'$
refines $\mathcal D$. Also, since the sets in $\cY'$ are connected and contained in
$\mathbf{D}\cap \mathbf{E}$ the assumption of the Lemma implies that $\cY'$ also refines $\cE$. 
Moreover  $\order ({\mathcal Y}')\leq n+2$
and no element of ${\mathcal Y}'$ meets both $f^{-1} \left(\frac{1}{3}
  \right)$ and
$f^{-1}\left(\frac{2}{3} \right)$.
For every $\Omega \in {\mathcal Y}'$ 
which intersects $f^{-1} \left(\frac{1}{3} \right)$, there is an
element $D \in \cD$ such that $\Omega \subset D$. We choose one such
$D$ and say that {\em $D$ marks $\Omega$}. Similarly if $\Omega$
intersects $f^{-1} \left( \frac{2}{3} \right)$ there is $E \in
\cE$ so that $\Omega \subset E$, we choose one such $E$ and say
that {\em $E$ marks $\Omega$}. We now modify elements of $\cD$ and
$\cE$: for each element $ D \in \mathcal D$, define 
$$
\tilde{D} \df \left( D \cap f^{-1} \left ( \left[0,1/3\right) \right) \right)\cup \bigcup_{D\text{ marks } \Omega} \Omega.
$$
Similarly we modify elements of $\cE$, defining 
$$
\tilde{E} \df \left( E \cap f^{-1} \left( \left( 2/3,1 \right]
  \right) \right) \cup \bigcup_{E \text{ marks } \Omega} \Omega.
$$
We refer to $\tilde{D}, \tilde{E}$ as {\em modified elements} of $\cD, \cE$. Finally define
$\mathcal Y$ as the collection of modified elements of 
${\mathcal D}$ and $\cE$ and the elements of ${\mathcal Y}'$
which do not meet 
$f^{-1} \left( \frac13 \right) $ or $ f^{-1} \left( \frac23 \right)
$. 
It is easy to see that $\mathcal X$ has the required properties. 
\end{proof}

\begin{lemma}\Name{lemma0036}
Let $Y$ be a metric space and let $\mathbf{D},\mathbf{E}_i, i\in
\crly{I}$ be open subsets which cover $Y$. Assume 
that the $\mathbf{E}_i$'s are disjoint, connected, and are uniformly
of $\on{asdim} \le \ell -1$.
Let $\cD$ be an open cover of $\mathbf{D}$ which is of bounded mesh
and $\on{ord}\cD\le \ell$. 
Then $Y$ has an open cover $\cY$ which refines the cover
$\cD\cup\set{\mathbf{E}_i:i\in\crly{I}}$, 
is of bounded
mesh and $\on{ord}\cY \le \ell+1$.
\end{lemma}
\begin{proof}
Using the assumption that $\mathbf{E}_i$ is uniformly of  $\on{asdim} \le \ell-1$ we find 
an open cover $\cE_i$ of $\mathbf{E}_i$ which is of uniformly bounded
mesh, such that $\on{ord}\cE_i\le \ell$ and
$\on{Leb} \cE_i> \on{mesh} \cD$. 
Let
$\mathbf{E}\defi\bigcup_{\crly{I}} \mathbf{E}_i$ and let  
$$\cE\defi\bigcup_{i \in \crly{I}}
\{\mathbf{E}_i \cap U: U \in \cE_i\}.
$$

Clearly it suffices to verify that the hypotheses of 
Lemma~\ref{p1}
are satisfied. Indeed, 
by assumption the cover $\cD$ is of bounded $\mesh$ and order, and 
$\cE$ is of bounded mesh because of the
uniform bound on $\mesh(\cE_i)$. We  also have that
$\order \cE\le \ell+1$ because of the bounds $\order \cE_i\le\ell+1$ 
and the fact that the $\mathbf{E}_i$ are disjoint. 
For 
the last condition, 
let a connected subset
$C\subset \mathbf{D}\cap\mathbf{E}$ which is contained in an element  
of $\cD$ be given. By the connectedness and disjointness of the
$\mathbf{E}_i$'s we conclude that there exists $i$ with  
$C\subset \mathbf{E}_i$. Because $\on{Leb} \cE_i>\mesh \cD$ we deduce
that  since $C$ is contained in an element of $\cD$ it must be  
contained in an element of $\cE_i$ and in turn, as it is contained in
$\mathbf{E}_i$, it must be contained in an element of  
$\cE$. 
%
\end{proof}

 \begin{proof}[Proof of Proposition \ref{p2}.] 
Proceeding inductively
in reverse order, for $k=m, \ldots, 1$ 
 we will construct 
 a uniformly bounded open cover ${\mathcal A}^k $ of $[{\mathcal
   A}]^k$ 
 such that
 $\order ({\mathcal A}^k) \leq m+1-k$ and 
  ${\mathcal A}^k$ refines the restriction of $\mathcal A$ to
  $[{\mathcal A}]^k$. The construction is  
 obvious for $k =m$. Namely, our hypothesis and Definition \ref{def:
   asdim} with $n=m-k=0$ mean that 
$[\mathcal{A}]^m$ has a cover of bounded mesh and  order 1, that is,
we can just  set ${\mathcal A}^{m}$ to be
 the connected components of $[{\mathcal A}]^{m}$. 
Assume that the construction
 is completed for $k+1$ and proceed to $k$ as follows. First notice
 that for two distinct $k$-intersections $A$ and $A'$ of $\mathcal A$
 the complements $A\smallsetminus [{\mathcal A}]^{k+1}$
 and $A' \smallsetminus [{\mathcal A}]^{k+1}$ are disjoint.
By Lemma \ref{lem: for p2}, 
we can cover $[{\mathcal A}]^k\smallsetminus[{\mathcal A}]^{k+1}$ by 
a collection $\set{\mathbf{E}_i:i\in\crly{I}}$ of disjoint connected open sets such that every
$\mathbf{E}_i$ is contained in a $k$-intersection of $\mathcal A$. In particular, the collection
$\set{\mathbf{E}_i:i\in\crly{I}}$ is uniformly of asdim $\le m-k$. We can therefore apply Lemma~\ref{lemma0036}
with the choices $Y=\br{\cA}^k, \mathbf{D}=\br{\cA}^{k+1}, \cD=\cA^{k+1}$, the collection $\set{\mathbf{E}_i:i\in\crly{I}}$, and
$\ell=m-k$, and obtain an open cover $\cY$ of $\br{\cA}^k$ of order $\le m-k+1$ that refines $\cD\cup\set{\mathbf{E}_i:i\in\crly{I}}$ and 
in particular, refines $\cA|_Y$. This completes the inductive step.

\end{proof}

\begin{proof}[Proofs of Proposition \ref{p3}.] 
For every $A \in \mathcal A$ choose
 a vertex $v^A_i$ of $\Delta_i$ so that $p_i(A)$ does not intersect 
 the face of $\Delta_i$ opposite to $v^A_i$. Let $Y \df \nerve(\cA)$
 and let $f : X \to X$ be the
 composition of 
a canonical map  $ X \to Y$ and a map 
$Y\to X$
which is  linear on each simplex of $Y$
and sends the vertex of $Y$ related to $A \in \mathcal A$
to the point $(v_1^A, v_2^A) \in X$. 
Take a point $x \in \partial \Delta_1 \times \Delta_2$. Then 
$p_1(x)$ belongs to a face $\Delta'_1$ of $\Delta_1$ and hence
for every $A \in \mathcal A$ containing $x$ we have that $v^A_1 \in \Delta'_1$.
Thus both $x$ and $f(x)$ belong to $\Delta'_1 \times \Delta_2$.
 Applying the same argument to
 $\Delta_1 \times \partial \Delta_2$ we get that
 the boundary $\partial X$ is invariant under $f$ and $f$ restricted
 to $\partial X$ is homotopic to the identity map of $\partial X$.
 If $\order ( {\mathcal A}) \leq \dim \Delta_1 + \dim \Delta_2$ then
  $\dim Y \leq \dim X -1$ and hence
 there is an interior point $a$ of $X$ not covered by $f(X)$. Take 
 a retraction $r : X \smallsetminus \{ a\} \to \partial X$. Then
 the identity map of $\partial X$ factors up to homotopy through 
 the contractible space $X$ which contradicts the non-triviality
 of the reduced homology of $\partial X$. 
\end{proof}

\bibliographystyle{alpha}
\def\cprime{$'$} \def\cprime{$'$} \def\cprime{$'$}

\end{document}